%% file: main.tex
\documentclass[10pt]{article}

\usepackage[
    letterpaper,
    textwidth=5.5in,
    textheight=8.5in,
    centering
]{geometry}

\input{preamble}

\title{Structural Visibility in Dynamical Systems on Hypergraphs: A Pattern Formation Perspective}

\author{Moise R. Mouyebe and Anthony M. Bloch}

\date{}

\begin{document}
\maketitle

\begin{abstract}
    Hypergraphs encode rich higher-order interactions, but not all structural information is equally accessible through the dynamics. By analyzing pattern-forming instabilities in reaction-diffusion systems on directed hypergraphs, this work develops a theory of structural visibility that characterizes which features of higher-order structure survive successive levels of dynamical reduction. It is established that higher-order structure is not automatically dynamically relevant. Linearization destroys most higher-order information.  Meanwhile,   nonlinear reduction recovers only specific higher-order marginals of the adjacency tensor, and projection along critical directions further filters what is dynamically visible. First, we show that the linearized dynamics depends on the hypergraph only through its first-tail-moment statistics, termed exposure. Consequently, exposure-equivalent hypergraphs are linearly indistinguishable in the sense that they exhibit identical dispersion relations, critical
eigenspaces, and instability thresholds. Next, we define a hierarchy of hyperedge tail-moments that captures progressively
detailed  co-participation, and we prove a nonlinear structural decomposition theorem describing how contractions of these tensors, termed packing effects,  influence the reduced amplitude dynamics. This leads to a visibility hierarchy in which successive asymptotic orders reveal increasingly richer structural information. More specifically, exposure governs linear onset, while packing effects control post-onset dynamics. Finally, we establish results on nonlinear distinguishability, characterizing when  linearly indistinguishable higher-order systems may exhibit different saturated
amplitudes, branch selection and pattern morphology. In addition, we formalize when higher-order systems become dynamically indistinguishable from pairwise systems near onset, leading to the notion of dynamical graph surrogacy. Numerical simulations support the theoretical predictions.
\end{abstract}


\vspace{0.5em}

\section{Introduction}
Many complex systems involve interactions that cannot be adequately represented by simple pairwise relationships alone \cite{battiston2020networks,battiston2021physics}. In social settings, interactions commonly take place within groups rather than between isolated pairs of individuals. Similarly, biological processes are frequently mediated by molecular complexes composed of multiple components interacting simultaneously \cite{klamt2009hypergraphs}. Systems involved in communication, information processing, and collective decision-making likewise rely on interactions among groups of agents instead of one-to-one exchanges  \cite{battiston2020networks,iacopini2019simplicial, bairey2016high,centola2018behavior}. These observations have motivated a growing interest in higher-order network models \cite{bianconi2021higher} capable of accurately representing such genuinely multiway interactions.

Hypergraphs provide a natural  framework for representing and describing such systems \cite{lucas2020multiorder,schaub2020random,millan2020explosive}. Unlike ordinary graphs, where interactions occur only between pairs of nodes, hypergraphs can capture interactions among arbitrary collections of nodes. This enables them to encode structural information inaccessible to pairwise network models. Hypergraph-based approaches have been applied to investigate a broad range of phenomena, including diffusion, synchronization, contagion, collective dynamics and behavior, and pattern formation \cite{turing1952chemical,nakao2010turing,carletti2020dynamical, battiston2020networks,carletti2020random}.

Despite the flourishing research activity in this area, the connection between higher-order structure and dynamical behavior is still not fully understood.  
The increased expressive power of hypergraphs raises the following fundamental
question. Although hypergraphs encode rich higher-order structure, dynamical systems do not necessarily respond to all structural information equally.
Different structural features may influence different aspects of the dynamics,
while other features may remain dynamically invisible. Consequently, the
presence of higher-order interactions does not automatically imply that all higher-order
information is observable through system behavior. This leads to the central question of the present work, namely \emph{which features of hypergraph structure are visible through system  dynamics?}

To address this question,  we study reaction-diffusion systems on directed hypergraphs, and we analyze diffusion-driven instabilities \cite{turing1952chemical,nakao2010turing} and pattern formation \cite{murray2003mathematical,cross1993pattern,hoyle2006pattern,maini2012turing} for a general class of nonlinear multiway interaction operators.
 Our framework is inspired by the DeepSet construction \cite{zaheer2017deep}, which characterizes permutation-invariant aggregation through symmetric pooling  followed by nonlinear activation. While originally developed in the machine learning context \cite{zaheer2017deep,lee2019set}, the DeepSet viewpoint provides a natural and analytically tractable mechanism for defining genuinely multiway  operators without reducing interactions to pairwise form. In our setting, each directed hyperedge aggregates the states of its tail nodes through a nonlinear activation applied to their mean state, hence preserving both directionality and higher-order coupling structure.

In this work, we develop a theory of structural visibility for reaction-diffusion systems on directed hypergraphs  near a codimension-one steady
bifurcation.
Our first main result shows that  the linearized dynamics depends only on first-tail-moment statistics, termed
\emph{exposure}, which yields an induced operator reminiscent of the graph
Laplacian. Consequently, hypergraphs that agree in their exposure statistics are linearly indistinguishable. Meaning, they exhibit identical dispersion relations, critical eigenmodes, and instability thresholds. In particular, all higher-order co-occurrence information inside hyperedges disappears at linear order. This provides a precise explanation for why graph-based surrogates can remarkably well reproduce instability onset despite neglecting higher-order interactions.

Beyond onset, nonlinear saturation reveals progressively richer co-occurrence
information through a hierarchy of hyperedge tail-moments, termed \emph{packing tensors}. We prove a nonlinear structural decomposition theorem that establishes that the reduced amplitude dynamics does not depend directly on these raw tensors themselves, rather on projected \emph{packing effects} obtained through contraction with the critical left and right eigenvectors. These projected quantities constitute the dynamically visible part of higher-order hypergraph structure near onset, and they govern nonlinear saturation, branch selection, and pattern morphology. 
This reveals a new mechanism absent in purely pairwise systems, namely, higher-order interactions can remain completely invisible at linear onset while re-emerging through nonlinear saturation. 

Finally, building on this principle, we introduce the notion of weak packing-equivalence, which compares only the higher-order information visible in the reduced dynamics, and we establish nonlinear distinguishability results based on differences in  packing effects, by which exposure-equivalent systems may become distinguishable only beyond onset despite exhibiting identical linear instability behavior. In addition, we formalize the notion of dynamical graph surrogacy  near onset. We show that a hypergraph-coupled system may remain dynamically equivalent to a graph-based system despite the presence of genuine higher-order interactions, provided the corresponding  packing effects vanish. 

The paper is structured as follows: 
Section~\ref{sec:section2} reviews background on  hypergraphs and higher-order directionality. Section~\ref{sec:section3} introduces the general framework of nonlinear multiway diffusion operators on directed hypergraphs. Section~\ref{sec:linear-onset} develops the linear theory and characterizes instability onset through exposure operators. Section~\ref{sec:nonlinear-saturation} develops the weakly nonlinear theory, introduces  packing effects, weak packing-equivalence, and establishes the nonlinear distinguishability and dynamical graph surrogacy results. Section~\ref{sec:section6} presents numerical experiments validating the theoretical predictions. Finally, Section~\ref{sec:section7} concludes with a discussion and outlook.

\section{Hypergraph Background and Directionality}\label{sec:section2}

In this section we recall basic notions from hypergraph theory and introduce the directed hypergraph formalism used throughout the paper. This section establishes the hypergraph objects used throughout the paper and fixes notation. We work with directed hypergraphs in which a set of tail nodes jointly influences a head node. This setting captures genuinely higher-order interactions while retaining a notion of directionality that is essential in many applications. Our goal is not to provide a comprehensive survey, but to establish a minimal and consistent set of definitions sufficient to formulate and analyze dynamical systems with higher-order interactions. We also describe a directification procedure that allows undirected hypergraphs to be embedded into the directed framework, ensuring that the theory developed here applies to both settings.

\subsection{Hypergraphs}

A hypergraph (undirected) is a generalization of a graph in which edges can connect more than two vertices simultaneously. Formally, a hypergraph is a pair $\mathcal{H}=(V,E)$, where $V=\{v_1,\dots, v_N\}$ is a finite set of vertices and $E \subseteq 2^V$ is a collection of hyperedges, each of which is a nonempty subset of vertices. The size of a hyperedge is the number of vertices it contains. Its \emph{rank} $r(\mathcal H):=\max_{e\in E}|e|$ is the size of the largest hyperedge, meanwhile the \emph{corank} $c(\mathcal H):=\min_{e\in E}|e|$ is the size of its smallest hyperedge.

An undirected hypergraph is said to be  \emph{$k$-uniform} if all its hyperedges are of size $k$, or equivalently if $r(\mathcal H) = c(\mathcal H) = k$. In particular, when $k=2$, the hypergraph reduces to a standard graph.

 Allowing hyperedges of arbitrary size enables the modeling of multiway interactions that cannot be decomposed into pairwise effects.

\subsection{Directed hypergraphs and B-hypergraphs}

To describe asymmetric interaction,  we work with directed hypergraphs.

\begin{definition}[Directed hyperedge]\label{def:directed-hyperarc}
A directed hyperedge (or \emph{hyperarc}) on $V$ is an ordered pair $e = (T(e),H(e))$ of disjoint nonempty subsets of $V$, where $T(e)\subset V$ is the \emph{tail} set and $H(e)\subset V$ is the \emph{head} set. We assume $T(e)\cap H(e)=\emptyset$ to avoid loops.
\end{definition}

In this paper,  we focus on a particular class known as B-hypergraphs \cite{gallo1993directed} which are sufficient for modeling the dynamics considered here.

\begin{definition}[B-arc and B-hypergraph]\label{def:B-hypergraph}
A hyperarc $e=(T(e),H(e))$ is called a \emph{B-arc} if $|H(e)|=1$. In this case we write $H(e)=\{v_i\}$ and denote the hyperarc as $e=(T\to v_i)$.
A directed hypergraph $\mathcal H=(V,E)$ is a \emph{B-hypergraph} if every $e\in E$ is a B-arc.
\end{definition}

\begin{definition}[$k$-uniform B-hypergraph]
    A B-hypergraph $\mathcal{H} = (V,E)$ is said to be $k$-uniform if all of its B-arcs have tail size $k$.
\end{definition}

Oftentimes, it is useful to assign \emph{weights} to each hyperedge in order to encode heterogeneous interaction strengths. This gives rise to notion of \emph{weighted (directed) hypergraph}, namely   a triple $\mathcal H=(V,E,w)$ where $(V,E)$ is a (directed) hypergraph and $w:E\to\mathbb R_{>0}$ assigns a positive weight to each hyperarc.

Note that we will sometimes refer to a node only by its index whenever the context is clear. For instance,  $e=(T\to v_i)$ will be sometimes denoted as $e=(T\to i)$.

\subsection{Order decomposition and uniform B-components}\label{subsec:uniform-components}

In a B-hypergraph, hyperarcs may have varying tail cardinalities. Since the dynamics will sum over interaction orders, it is convenient to group hyperarcs by order.

\begin{definition}[Order-$k$ component]\label{def:order-k-component}
Let $\mathcal H=(V,E,w)$ be a B-hypergraph. For each $k\in \{1, \dots, r(\mathcal{H})-1 \}$, define

\begin{equation}\label{eq:order-k-component}
    E_k:=\{e=(T\to i)\in E:\; |T|=k\},
\qquad \mathcal H_k:=(V,E_k,w|_{E_k}).
\end{equation}

If $E_k\neq\emptyset$, then $\mathcal H_k$ is a $k$-$uniform$ B-hypergraph called the {$k$-uniform B-component} of $\mathcal H$. Furthermore, the $\mathcal{H}_k$ form a partition of $\mathcal{H}$.
\end{definition}

 In this paper, we sometimes  use $\rho$ to denote the hypergraph rank, namely  the maximal hyperedge \emph{size} (tail+head), so tails have size at most $\rho-1$.

\subsection{Adjacency tensors for uniform  B-hypergraphs}\label{subsec:adjacency-tensor}

\begin{definition}[Adjacency tensor]\label{def:adjacency-tensor}
Let $\mathcal H_k=(V,E_k,w)$ be a $k$-uniform B-hypergraph. The adjacency tensor of $\mathcal{H}_k$ is the order-$(k+1)$ tensor $\textsf{A} \in \mathbb{R}^{N\times N \times \cdots \times N}$ defined as follows

\begin{align}\label{eq:adjacency-tensor-def}
    \textsf{A}^{(k)}_{ij_1\cdots j_{k}} &= \begin{cases}
\frac{w(e)}{k!} & \text{if } ~e = \left(\left\{v_{j_1},\cdots, v_{j_{k}} \right\},v_i\right) \in E \\ \\
0 & \text{otherwise } 
\end{cases}
\end{align}    
\end{definition}

\begin{remark}\label{rem:normalization-k-factorial}
The normalization used in \eqref{eq:adjacency-tensor-def} is convenient when summing over ordered tuples $(v_{j_1},\dots,v_{j_k})$ in the dynamical equations. Each unordered tail set $\{v_{j_1},\dots,v_{j_k}\}$ appears $k!$ times under permutations, and the factor $1/k!$ exactly cancels this multiplicity. This makes expressions invariant to tail ordering while allowing us to use standard tensor-index summations. 

Furthermore, we note that in general adjacency tensor $\textsf{A}^{(k)}$  of a B-hypergraph is not \emph{symmetric}. However by construction, it is symmetric in the tail indices $j_1,\dots,j_k$, namely the  $k$-order slices $\textsf{A}^{(k)}_{i:\cdots :}$ are symmetric.
\end{remark}

\subsection{Degrees for directed B-hypergraphs}\label{subsec:degrees}

For directed graphs, \emph{in-degree} and \emph{out-degree} distinguish incoming from outgoing interactions. In a B-hypergraph, the analogous quantities count incoming and  outgoing hyperarcs adjusted for their weights.

\begin{definition}[in-degree and out-degree]\label{def:degrees_in--and--out}
Let $\mathcal H_k=(V,E_k,w)$ denote  a $k$-uniform B-hypergraph.
The (weighted) \emph{in-degree} and \emph{out-degree} of vertex $v_i$ and  are defined as follows:
\begin{align}
    d_\mathrm{in}^{(k)}(v_i) &\coloneqq \sum_{e=(T\to v_i)\in E_k} w(e) \label{def:in-degree}\\
d_\mathrm{out}^{(k)}(v_i) &\coloneqq  \sum_{e=(T\to v_h)\in E_k:\; v_i\in T} w(e) \label{def:out-degree}
\end{align}
\end{definition}

The intuitive interpretation is that  $d_\mathrm{in}^{(k)}(v_i)$ measures the total incoming order-$k$ interaction strength into head vertex  $v_i$, whereas 
$d_\mathrm{out}^{(k)}(v_i)$ measures how strongly vertex $v_i$ participates in tails of order-$k$ interactions.

\subsection{Directification}\label{subsec:directification}

Our theory is developed for directed B-hypergraphs. To ensure that it also applies to undirected hypergraphs, we introduce a canonical procedure that embeds an undirected hypergraph into the space of B-hypergraphs by replacing each undirected hyperedge by a set of B-arcs.

\begin{definition}[B-arc set associated with an undirected hyperedge]\label{def:B-arc-set}
Let $\mathcal H=(V,E,w)$ be an undirected weighted hypergraph and let
$e=\{v_1,\dots,v_p\}\in E$ be a hyperedge of size $p\ge2$.
For each $v_i\in e$, define $e_{\hat{i}}:=e\setminus\{v_i\}$ and set
\begin{equation}
    B_{\mathrm{arc}}(e):=\{(e_{\hat{i}}\to v_i):\; v_i\in e\}
\end{equation}
\end{definition}

\begin{definition}[B-embedding / Directification]\label{def:B-embedding}
Let $\mathcal H=(V,E,w)$ be an undirected weighted hypergraph. Its \emph{B-embedding} (or \emph{directification}) is the weighted B-hypergraph
\begin{equation}
    \widetilde{\mathcal H}:=(V,\widetilde E,\widetilde w),
\qquad
\widetilde E:=\bigcup_{e\in E} B_{\mathrm{arc}}(e)
\end{equation}
with weights inherited from $\mathcal H$ via
\begin{equation}
    \widetilde w\big(e\setminus\{v\}\to v\big):= w(e)
\quad \text{for each } e\in E,\;\text{ and }\; v\in e
\end{equation}
\end{definition}

The B-embedding replaces an undirected group interaction among $p$ vertices by the $p$ directed influences in which each vertex acts as a head receiving influence from the remaining $p-1$ vertices. This choice preserves the permutation symmetry of the original undirected hyperedge while producing a directed structure compatible with our theory.

\subsubsection{Directification Example}\label{ex:directification}

Let $V=\{1,2,3,4\}$ and consider the  $3$-uniform undirected hypergraph   with   hypergraphs $e_1=\{2,3,4\}$ and $e_2=\{1, 2,4\}$
with weight $w(e_1)=\omega = w(e_2)$ shown in Figure~\ref{fig:3-uniform-hypergraph_example}.
Then the B-arc set associated with $e_1$ is $B_{\mathrm{arc}}(e_1)=\{(\{3,4\}\to 2),(\{2,4\}\to 3),(\{2,3\}\to 4)\}$ and each resulting B-arc has weight $\omega$ as shown in Figure~\ref{fig:barc_e1}.  Similarly  we have $B_\mathrm{arc}(e_2) =\{(\{2,4\}\to 1),(\{1,4\}\to 2),(\{1,2\}\to 4)\}$.

\tikzstyle{vertex} = [color=black!60,fill,shape=circle]
\tikzstyle{edge} = [line width=.1pt]

\begin{figure}[H]
\centering

\resizebox{0.35\linewidth}{!}{
\begin{tikzpicture}

    \node[vertex,text=white,scale=1] (v1) at (0,2) {1};
    \node[vertex,text=white,scale=1] (v2) at (2,2) {2};
    \node[vertex,text=white,scale=1] (v3) at (2,0) {3};
    \node[vertex,text=white,scale=1] (v4) at (0,0) {4};

    \begin{scope}[fill opacity=0.8]

    \draw[fill=red!60!magenta,opacity=.4]
     ($(v2)+(0,0.5)$) 
        to[out=0,in=90] 
        ($(v3) + (0.5,0)$)
        to[out=270,in=0]
        ($(v4) + (0,-0.5)$)
        to[out=180,in=270]
        ($(v4) + (-0.5,0)$)
        to[out=90,in=180]
        ($(v4) + (0,0.5)$)
        to[out=0,in=175]
        ($(v4)+(1.2,0.4)$)
        to[out=0,in=270]
        ($(v4)+(1.5,0.7)$)
        to[out=90,in=270] 
        ($(v2)+(-0.5,0)$)
        to[out=90,in=180] 
        ($(v2)+(0,0.5)$);

    \draw[fill=green!100,opacity=.4]
     ($(v2)+(0,0.5)$) 
        to[out=180,in=90] 
        ($(v1) + (-0.5,0)$)
        to[out=270,in=180]
        ($(v4) + (0,-0.5)$)
        to[out=0,in=270]
        ($(v4) + (0.5,0.1)$)
        to[out=90,in=270] 
        ($(v4)+(.45,1)$)
        to[out=90,in=180] 
        ($(v2)+(0,-0.5)$)
        to[out=0,in=270] 
        ($(v2)+(0.55,0)$)
        to[out=90,in=0] 
        ($(v2)+(0,0.5)$);

    \end{scope}

\end{tikzpicture}
}

\caption[A 3-uniform undirected hypergraph on four nodes]{A $3$-uniform hypergraph on four nodes with hyperedges 
$e_1=\{2,3,4\}$ and $e_2=\{1,2,4\}$.}
\label{fig:3-uniform-hypergraph_example}

\end{figure}
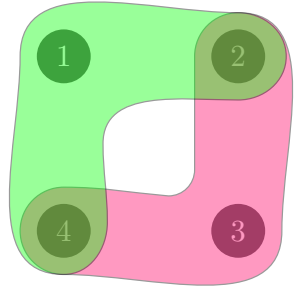


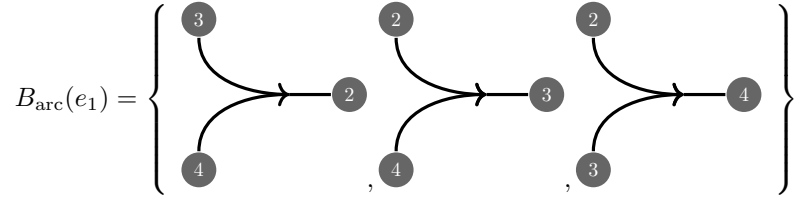
\begin{figure}[H]
\centering

\resizebox{0.75\linewidth}{!}{%
$
B_{\mathrm{arc}}(e_1)=
\left\{
\begin{array}{c}

\begin{tikzpicture}

    \node[vertex,text=white,scale=.75] (v3) at (0,2) {3};
    \node[vertex,text=white,scale=.75] (v4) at (0,0) {4};
    \coordinate (temp) at (1.2,1) {};
    \node[vertex,text=white,scale=.75] (v2) at (2,1) {2};

    \draw[-, line width=1.2, shorten <=.1mm] 
        (v3) to[out=270,in=180] (temp);

    \draw[->, line width=1.2, shorten <=.1mm]
        (v4) to[out=90,in=180] (temp);

    \draw[-, line width=1.2, shorten >=.1mm]
        (temp) to[out=0,in=180] (v2);

\end{tikzpicture},

\begin{tikzpicture}

    \node[vertex,text=white,scale=.75] (v2) at (0,2) {2};
    \node[vertex,text=white,scale=.75] (v4) at (0,0) {4};
    \coordinate (temp) at (1.2,1) {};
    \node[vertex,text=white,scale=.75] (v3) at (2,1) {3};

    \draw[->, line width=1.2, shorten <=.1mm] 
        (v2) to[out=270,in=180] (temp);

    \draw[->, line width=1.2, shorten <=.1mm]
        (v4) to[out=90,in=180] (temp);

    \draw[-, line width=1.2, shorten >=.1mm]
        (temp) to[out=0,in=180] (v3);

\end{tikzpicture},

\begin{tikzpicture}

    \node[vertex,text=white,scale=.75] (v2) at (0,2) {2};
    \node[vertex,text=white,scale=.75] (v3) at (0,0) {3};
    \coordinate (temp) at (1.2,1) {};
    \node[vertex,text=white,scale=.75] (v4) at (2,1) {4};

    \draw[->, line width=1.2, shorten <=.1mm] 
        (v2) to[out=270,in=180] (temp);

    \draw[->, line width=1.2, shorten <=.1mm]
        (v3) to[out=90,in=180] (temp);

    \draw[-, line width=1.2, shorten >=.1mm]
        (temp) to[out=0,in=180] (v4);

\end{tikzpicture}

\end{array}
\right\}
$
}

\caption[The B-arc set $B_{\mathrm{arc}}(e_1)$ associated with the hyperedge
$e_1=\{2,3,4\}$]{The B-arc set associated with the hyperedge
$e_1=\{2,3,4\}$, obtained by selecting each node in turn as the head node while the remaining nodes form the tail set.}
\label{fig:barc_e1}
\end{figure}

\begin{remark}[Relevance of directification]
All results proved for directed B-hypergraphs apply to undirected hypergraphs by first applying the B-embedding and then interpreting the resulting directed dynamics as the natural directed representation of the original undirected group interactions.
\end{remark}

\subsection{Symmetric B-hypergraphs}

\begin{definition}[Complement B-arc set]
    Let $\mathcal{H} = (V,E,w)$ be a uniform B-hypergraph, and consider the B-arc $e = (T(e),h) \in E$. We define the \emph{complement B-arc set} $\bar{B}_{arc}(e)$ of $e$ as follows:
    \begin{equation}\label{complement_B-arc}
        \bar{B}_{arc}(e) = \left\{ \left(  T_{\hat{v}}(e)\cup \{h\} ,v\right) ~\rvert~ v \in T(e) \right\}
    \end{equation}
    where $T_{\hat{v}}(e) = T(e)\setminus \{v\}$
    \end{definition}

The notion of \emph{complement arc set} defined above allows us to give precise definition of  \emph{symmetric B-hypergraphs} which generalizes the notion of \emph{bidirectional graphs}.

\begin{definition} [Symmetric Uniform B-hypergraph]
    Let $\mathcal{H} = (V,E,w)$ be a uniform B-hypergraph. $\mathcal{H}$ is said to be \emph{symmetric} if its hyperarc set $E$ is \emph{closed under B-arc set complementation}. More specifically,  $\mathcal{H}$ is symmetric if the following holds:
    \begin{equation}
        \bar{B}_{arc}(e) \subset E ~ \text{ whenever} ~ e \in E
    \end{equation}
\end{definition}

\begin{proposition}
    Let $\mathcal{H} = (V,E,w)$ be a symmetric uniform  B-hypergraph. The adjacency tensor (see \ref{eq:adjacency-tensor-def}) of $\mathcal{H}$ is symmetric.
\end{proposition}
\begin{proof}
    It follows  by construction.
\end{proof}

\begin{proposition}
    Let $\mathcal{H} = (V,E,w)$ be a uniform undirected hypergraph. Its B-embedding $\tilde{\mathcal{H}}$ is a symmetric  B-hypergraph.
\end{proposition}
\begin{proof}
    The map 
    \[  \imath \colon \left\{\text{Undirected Hypergraphs}\right\} \hookrightarrow \left\{ \text{Symmetric  B-hypergraphs} \right\} \]

   that sends an undirected hypergraph $\mathcal{H}$ to its B-embedding $\tilde{\mathcal{H}} = \imath(\mathcal{H})$ is an injection that preserves symmetry.
\end{proof}

\section{A Unifying Framework for Dynamical Systems on Directed Hypergraphs}\label{sec:section3}

We now introduce a general framework for defining and analyzing dynamical systems on
directed B-hypergraphs. The framework accommodates higher-order directed
interactions between groups of nodes, supports multi-dimensional node states,
and it includes a wide and flexible  class of coupling mechanisms. It is sufficiently general
to recover many classical network dynamical systems as special cases, while
providing a systematic extension to directed hypergraph settings.

\subsection{Dynamical systems on hypergraphs}

We consider a system of $N$ identical dynamical units (nodes or agents), indexed by $i\in\{1,\dots,N\}$. The state of node $i$ is denoted by $ x_i(t)\in\mathbb R^d$, 
where $d\ge1$ is the dimension of the local state space. In isolation, each node evolves according to the same intrinsic dynamics
\begin{equation}\label{eq:intrinsic-dynamics}
\dot x_i = f(x_i)
\end{equation}
where $f:\mathbb R^d\to\mathbb R^d$ is assumed to be sufficiently smooth.
The joint state of the system is
\[
\mathbf{x}=(x_1,\dots,x_N)\in(\mathbb R^d)^N\cong\mathbb R^{Nd}
\]

Let $\mathcal{H} = (V,E,w)$ be a B-hypergraph of rank $\rho \ge 2$  modeling the interaction of these $N$ agents. As discussed in Definition \ref{def:order-k-component}, $\mathcal{H}$ can be partitioned into uniform B-components $\{\mathcal{H}_k\}_{k=1}^{\rho-1}$ collecting all hyperarcs of tail size $k$. Each uniform component $\mathcal{H}_k$ is encoded by an order-$(k+1)$ adjacency tensor $\textsf{A}^{(k)}$ with entries $\textsf{A}^{(k)}_{ij_1 \cdots j_k}$ as defined  in Definition \ref{def:adjacency-tensor}. The first mode corresponds to the head index, while the remaining $k$ modes correspond to the tail indices.

\begin{definition}[Dynamical system on directed  hypergraphs]\label{def:hypergraph-dynamics}
A \emph{dynamical system on a directed hypergraph} $\mathcal H$ is a system of ordinary differential equations of the form
\begin{equation}\label{eq:hypergraph-dynamics}
\dot x_i
=
f(x_i)
+
\sum_{k=1}^{\rho-1}
\sum_{j_1,\dots,j_k}
\mathsf A^{(k)}_{i j_1\cdots j_k}
\,\phi_k(x_{j_1},\dots,x_{j_k},x_i),
\qquad i=1,\dots,N
\end{equation}
where for each $k$, the \emph{coupling function}
\[
\phi_k:(\mathbb R^d)^k\times\mathbb R^d \to \mathbb R^d
\]
is a sufficiently smooth function that specifies how the states of the $k$ tail nodes jointly influence the head node.
\end{definition}

The following proposition gives an equivalent formulation of the dynamical system \eqref{eq:hypergraph-dynamics} under additional symmetry assumption. This formulation turns out to be more  for computationally tractable. 

\begin{proposition}[Hyperedge formulation]\label{prop:hyperedge-formulation}
Assume that, for each order \(k\) the coupling function \(\phi_k\) is symmetric in its \(k\) tail arguments. Then the tensor formulation given in \eqref{eq:hypergraph-dynamics}

is equivalent to the hyperedge-sum formulation
\begin{equation}\label{eq:hyperedge-sum-formulation}
\dot x_i
=
f(x_i)
+
\sum_{k=1}^{\rho-1}
\sum_{e=(T\to i)\in E_k}
w(e)\,\phi_k\big((x_j)_{j\in T},x_i\big)
\end{equation}
where \(E_k\) denotes the set of \(k\)-hyperedges.
\end{proposition}

\begin{proof}
Fix \(i\) and \(k\). For any hyperedge \(e=(T\to i)\) with \(|T|=k\), the tail set \(T\) appears in \eqref{eq:hypergraph-dynamics} through the \(k!\) permutations of its elements. Since
$
A^{(k)}_{i j_1\cdots j_k}=\frac{w(e)}{k!}
$
for each such ordering (see \eqref{eq:adjacency-tensor-def}), the total contribution of \(e\) to \eqref{eq:hypergraph-dynamics} is
$
\frac{w(e)}{k!}
\sum_{\pi\in S_k}
\phi_k(x_{j_{\pi(1)}},\dots,x_{j_{\pi(k)}},x_i),
$
where $\pi \in S_k$ is a permutation on $k$ symbols. By symmetry of \(\phi_k\) in the tail arguments, each term in the sum is identical, so this reduces to
$
\frac{w(e)}{k!}\cdot k!\,\phi_k\big((x_j)_{j\in T},x_i\big)
=
w(e)\,\phi_k\big((x_j)_{j\in T},x_i\big).
$
Summing over all \(k\)-hyperedges with head \(i\) yields \eqref{eq:hyperedge-sum-formulation}.
\end{proof}


\subsection{Aggregator-Mixer coupling structure}

Although there is a lot of flexibility in choosing the coupling functions $\phi_k$ in \eqref{eq:hypergraph-dynamics}, many applications share a common structure in which the states of the tail nodes are first aggregated into a collective signal, then mixed into the head node dynamics. To capture this structure, we propose the following flexible  class of coupling functions:

\begin{align}\label{eq:coupling-function-class}
    \phi_k(x_{j_1},\cdots,x_{j_k},x_i) =  M_k \Theta_k\big(  \alpha_k \psi_k(x_{j_1},\cdots,x_{j_k})   + \beta_k h_k(x_i) \big), 
\end{align}

where:

\begin{itemize}
    \item $\psi_k:(\mathbb{R}^d)^k \to \mathbb{R}^d$ is a permutation-invariant \emph{aggregator} or \emph{tail function}.

    \item $\Theta_k :\mathbb{R}^d \to \mathbb{R}^d$ is the \emph{mixer}. This function could be linear or nonlinear.

    \item $h_k: \mathbb{R}^d \to \mathbb{R}^d$ is  the \emph{head function}.

    \item $M_k \in \mathbb{R}^{d \times d}$ is the mixing matrix.

     \item $\alpha_k, \beta_k \in [-1,1]$ are \emph{gates}.
\end{itemize}

From a modeling standpoint, there is a lot of versatility in the use of the mixing matrix $M_k$. For instance,

\begin{enumerate}
    \item It can encode the \emph{coupling strength parameter} whenever it is a multiple of the identity matrix, namely $M_k = \lambda_k I_d$, for some $\lambda_k \in \mathbb{R}$.

    \item It can model a \emph{channel-wise selector operator} that controls cross-channel coupling as in \cite{mouyebe2025coupling}.

    \item It  can encode the \emph{species diffusion matrix} in reaction-diffusion systems. \emph{This is the point of view we will adopt  throughout this paper, where $M_k$ will be the diagonal matrix of diffusion coefficients}.

    \item Lastly, it can be factored into a product of relevant matrices to encode any combination of the above accordingly.
\end{enumerate}

\subsection{Relation to existing models and unifying perspective}

The framework
\eqref{eq:hypergraph-dynamics}--\eqref{eq:coupling-function-class}
subsumes a variety of classical dynamical systems as special cases while providing a systematic mechanism for extending them to directed hypergraph interactions.

In the rank-$2$ case (i.e.\ $\rho=2$), the adjacency tensor reduces to the standard adjacency matrix associated with directed graphs. Choosing linear aggregator and mixing functions,
\begin{equation}
    \psi(x)=x, \qquad  \Theta=\mathrm{Id}
\end{equation}
the framework recovers classical diffusive or consensus-type coupling on directed networks
\cite{olfati2004consensus,ren2005consensus,arenas2008synchronization}.

For scalar phase variables, suitable choices of linear aggregator, head function, and trigonometric mixer recover Kuramoto-type dynamics \cite{kuramoto2005self,acebron2005kuramoto,strogatz2000kuramoto}. For example, taking
\begin{equation}
    \Theta(x)=\sin(x)
\end{equation}
in the rank-$2$ setting yields the classical Kuramoto model on directed graphs, while the case $\rho>2$ produces higher-order generalizations of Kuramoto dynamics on directed hypergraphs.  More details are provided in section~ \ref{app:trigonometric-coupling} in the appendix.

The framework also encompasses tensor-based multilinear hypergraph systems previously studied in \cite{chen2021controllability}. In particular, choosing
\begin{equation}
    \Theta_k=\mathrm{Id}, \qquad h_k\equiv 0
\end{equation}
together with the multiplicative (Hadamard product) aggregator
\begin{equation}
    \psi_k(x_{j_1},\dots,x_{j_k})= x_{j_1}\odot x_{j_2}\odot \cdots \odot x_{j_k}
\end{equation}
recovers and extends multilinear tensor interactions to  directed hypergraphs.

When combined with local reaction terms $f$, diagonal mixing matrices $M_k$, linear mixers $\Theta_k=\mathrm{Id}$, and diagonal head functions, the framework also recovers   reaction-diffusion systems on network structures \cite{nakao2010turing,muolo2023turing}. 

In general, the proposed aggregator-mixer formulation provides a flexible architecture for modeling nonlinear multiway interactions. Appropriate choices of aggregator, mixer, and head functions allow the framework to capture threshold effects, saturation phenomena, cooperative interactions, and other nonlinear mechanisms arising in biological, social, and physical systems.
The few examples given above illustrate that the proposed framework is not tied to a single dynamical model, but instead it  provides a unifying  perspective on higher-order dynamical systems. See Appendix~\ref{app:aggr-mixer} for additional information.

\subsection{Diffusive case and Laplacian dynamics}

We now specialize the general framework above to the diffusive coupling regime, and we show how a natural Laplacian-type structure emerges at the tensorial level. This case provides the direct higher-order analogue of classical Laplacian-driven dynamics on graphs. It plays a central role in diffusion-induced instability analysis, and it clarifies the structural origin of the operators appearing in the subsequent linear and weakly nonlinear analyses.

\subsubsection{Diffusive coupling }
We say that the coupling function in~\eqref{eq:hypergraph-dynamics} is \emph{diffusive} if for each interaction order $k$, it  takes the form

\begin{equation}\label{eq:diffusive-coupling}
 \phi_k(x_{j_1},\dots,x_{j_k},x_i)
=
 M_k
\Big(
\psi_k(x_{j_1},\dots,x_{j_k})
-
\psi_k^{\Delta}(x_i)
\Big).
\end{equation}
This corresponds to choosing  $\Theta_k=\mathrm{Id}$, and a \emph{diagonal head function} $h_k(x_i) = \psi_k^{\Delta}(x_i) \coloneqq  \psi_k(x_i,\cdots,x_i)$ in \eqref{eq:coupling-function-class} for some smooth and permutation-invariant aggregator $\psi_k$. The defining feature of~\eqref{eq:diffusive-coupling} is that the coupling depends only on differences between aggregated tail states and the head state, ensuring that the interaction vanishes identically  on homogeneous configurations \eqref{eq:sync-manifold}.

\begin{definition}
    The synchronization manifold $\mathcal{S}$ of  system \emph{(\ref{eq:hypergraph-dynamics})} is the subspace of  the joint state space along which the states of all agents coincide. More specifically,   
    \begin{align}\label{eq:sync-manifold}
        \mathcal{S} &\coloneq \big\{ (x_1,\cdots,x_N) \in \mathbb{R}^{Nd} ~\rvert~ x_i = x_j, ~\forall~i,j = 1,\cdots, N  \big\} 
    \end{align}
\end{definition}

Note that coupling \eqref{eq:diffusive-coupling} vanishes on the synchronization  manifold \eqref{eq:sync-manifold}. In fact, this is one of the most important properties of \emph{diffusive coupling}.

\subsubsection{Laplacian tensor formulation}\label{subsec:laplacian-tensor-formulation}

Substituting \eqref{eq:diffusive-coupling} into \eqref{eq:hypergraph-dynamics}, we get  a higher-order reaction-diffusion dynamics involving the Laplacian tensor as follows

\begin{align}\label{eq:reaction-diffusion-on-hypergraph}
        \dot{x}_i  
         &= f(x_i) + \sum_{k=1}^{\rho-1}{ \sum_{j_1,\cdots, j_k}{ \mathsf{A}^{(k)}_{ij_1\cdots j_k} M_k \left (\psi_k(x_{j_1},\cdots,x_{j_k}) - \psi_k^{\Delta}(x_i) \right) } }      \notag\\
         &= f(x_i) - \sum_{k=1}^{\rho-1}{ \sum_{j_1,\cdots, j_k} \left (d_\mathrm{in}^{(k)}(v_i)\mathsf{\delta}_{ij_1\cdots j_k} - \mathsf{A}^{(k)}_{ij_1\cdots j_k} \right) M_k\psi_k(x_{j_1},\cdots,x_{j_k}) }          \notag\\
        &= f(x_i) - \sum_{k=1}^{\rho-1}{ \sum_{j_1,\cdots, j_k} \mathsf{L}^{(k)}_{ij_1\cdots j_k}M_k\psi_k(x_{j_1},\cdots,x_{j_k}) }        
\end{align}

where,   $d_{in}^{(k)}(v_i)$ is the in-degree of the head node $v_i$ as in Definition \ref{def:degrees_in--and--out}, $~\mathsf{\delta}$ is the order-$(k+1)$ Kronecker  diagonal tensor defined as follows

\begin{align}\label{eq:kronecker-tensor}
    \mathsf{\delta}_{i j_1 \cdots j_k} := 
\begin{cases}
1 & \text{if } j_1=\cdots=j_k = i,\\
0 & \text{otherwise}.
\end{cases}
\end{align}

and, $\mathsf{L^{(k)}} = \mathsf{D}^{(k)} - \mathsf{A}^{(k)}$ is the order-$(k+1)$ \emph{Laplacian tensor} associated with the uniform B-component $\mathcal{H}_k$ at interaction order $k$. $\mathsf{D}^{(k)}$ is the order-$(k+1)$ diagonal \emph{in-degree tensor} at interaction order $k$ defined as follows

\begin{align}\label{eq:in_Degree-tensor}
    \mathsf{D}^{(k)}_{i j_1 \cdots j_k} := 
\begin{cases}
d_\mathrm{in}^{(k)}(v_i) & \text{if } j_1=\cdots=j_k = i,\\
0 & \text{otherwise}.
\end{cases}
\end{align}

\subsection{Model specification }\label{sec:model-specification}

DeepSet models \cite{zaheer2017deep}  provide a natural framework for modeling permutation-invariant multiway interactions through a combination of elementwise transformations and symmetric pooling operations. The general structure of such aggregators takes the form

\begin{align}\label{eq:DeepSet-architecture}
    \psi_k = \sigma_k \circ \text{mean} \circ \lambda_k
\end{align}

where $\lambda_k : \mathbb{R}^d \to \mathbb{R}^d$ denotes an embedding transformation and $\sigma_k : \mathbb{R}^d \to \mathbb{R}^d$ represents an elementwise activation.

Throughout this work, we set the embedding map to the identity transformation, $\lambda_k = \mathrm{Id}$, and we concentrate on the simplified mean-pooling followed by an elementwise activation structure as follows

\begin{align}\label{eq:deepSet-aggregator-model}
    \psi_k(x_{j_1}, \ldots, x_{j_k}) = \sigma_k\left(\frac{1}{k}\sum_{r=1}^k x_{j_r}\right)
\end{align}

with $\sigma_k$ chosen as a \emph{sigmoid}, \emph{tanh}, \emph{Hill-type} nonlinearity, or any other sufficiently smooth \emph{nonlinear} saturating activation. This particular choice preserves the permutation-invariant properties characteristic of DeepSet-type models, while isolating the role of nonlinear activation in shaping diffusion and pattern formation, which is the focus of our subsequent analysis.

Under this specification, the coupling term takes the form

\begin{align}\label{eq:deepSet-coupling-function}
    \phi_k(x_{j_1}, \ldots, x_{j_k}, x_i) = M_k\left[\sigma_k\left(\frac{1}{k}\sum_{r=1}^k x_{j_r}\right) - \sigma_k(x_i)\right]
\end{align}

which will be the form used for the analysis   throughout the remainder of the paper.
We note that the activation $\sigma$ and the mixing matrix $M$ are fixed and independent of the interaction order $k$, and all  $k$-dependence enter only through the hypergraph structure.

The framework introduced in this section separates structure (encoded by the tensors $\mathsf A^{(k)}$) from dynamics (encoded by $f$, $\sigma$, and $M$). This separation will allow us to identify which contractions of $\mathsf A^{(k)}$ control linearized dynamics near homogeneous equilibria in \eqref{eq:sync-manifold}, and which higher-order contractions enter nonlinear saturation. In particular, in Section~\textcolor{red}{\ref{sec:linear-onset}} and Sec~\textcolor{red}{\ref{sec:nonlinear-saturation}} we will show that different dynamical regimes are associated with different tensorial marginals of the same underlying hypergraph structure.

\section{Linearization and Pattern Onset}\label{sec:linear-onset}

In this section,  we derive the linearization of the hypergraph reaction-diffusion system \eqref{eq:reaction-diffusion-on-hypergraph}  about a homogeneous equilibrium, namely one that belongs to the synchronization manifold \eqref{eq:sync-manifold}, and  we show that it depends only on the \emph{first tail moment}  statistics of the underlying hypergraph structure. This yields a  Laplacian operator  governing linear stability and reminiscent of graph settings. It also establishes a rigorous validity boundary for graph-based surrogates, and  provides a rigorous justification for why clique-expansion reductions  are quite successful at the level of pattern formation onset. This result forms the basis for all subsequent analysis.

We start with the following useful definition and lemma. The choice of the adjective "canonical" in Definition~ \ref{def:l-pointed-CHOM} is justified by the fact the object therein described  arises naturally from the dynamics as we will see later.

\begin{definition}[Canonical Higher Order Matrix]\label{def:l-pointed-CHOM}
    Let $\mathsf{T} \in I_0\times I_1 \times \cdots \times I_k$ be an order-($k+1$) tensor, where $I_r = \mathbb{R}^N$. For $l \in \{1,\cdots, k\}$, we define the $k$-th order l-pointed Canonical Higher Order Matrix or $\hat{l}$-CHOM of $\mathsf{T}$ as the following tensor contraction using tensor notation in \cite{kolda2009tensor}
    \begin{align}
        T^{\hat{l}} &= \langle \mathsf{T},\mathsf{1}_{k-1} \rangle_{ \{1,\cdots, l-1,l+1,\cdots, k\}}  \\
        &= \mathsf{T} \bar{\times}_{-\{0,l \}} \left\{\mathbf{1} \right\}
    \end{align}
where  $\mathbf{1}$ is the $N$-dimensional vector of all-ones, and $\mathsf{1}_{k-1} = \mathbf{1}^{\otimes (k-1)}$ is the order-(k-1) tensor of all-ones. 
\end{definition}

\begin{lemma}[Symmetry and Contraction Invariance]\label{lem:symmetry-and-contraction-invariance}
    Let $\mathsf{T} \in \mathbb{R}^{n_0\times n^k}$  be an order-($k+1)$ tensor written as $\mathsf{T}_{ij_1\cdots j_k}$ in coordinates, where $i \in \{1,\cdots,n_0\}$ and $j_l \in \{1,\cdots,n\}$ for all $1\le l \le k$. Assume $\mathsf{T}$ is symmetric in the last $k$ tail indices, namely for every permutation $\tau \in S_k$, we have $\mathsf{T}_{ij_1\cdots j_k} = \mathsf{T}_{ij_{\tau(1)}\cdots j_{\tau(k)}}$. Then the $\hat{l}$-CHOM of $\mathsf{T}$ is independent of $l$. More specifically, for any $l, l' \in \{1,\cdots,k\}$, we have $T^{\hat{l}}_{ij} = T^{\hat{l}'}_{ij}~$ for all $i,j$.
 \end{lemma}

\begin{proof}
    See Appendix~\ref{app:der-lem-sym-and-contraction-invariance}
\end{proof}

Lemma~\ref{lem:symmetry-and-contraction-invariance} above  allows us to unambiguously define the notion of tail moments of tail-symmetric tensors. As we will see in this section and the next, these tail moments play a fundamental role in the mechanism of pattern formation, both at instability onset and in the subsequent nonlinear regime.

\subsection{Structural visibility hierarchy}

\begin{definition}[$p$-th tail moment tensor]\label{def:pth-tail-moment}
    Let $\mathcal H_k=(V,E_k,w)$ be a $k$-uniform B-hypergraph. For an integer $p \in \{ 0,1, \cdots, k\}$, the $p$-th tail moment tensor is the order-(p+1) tensor
    \begin{align}\label{eq:pth-tail-moment-tensor}
\mathsf{T}^{(k,p)}_{ij_1\cdots j_p} \coloneqq 
\sum_{\substack{e=(T\rightarrow i)\in E_k, \\ \{j_1,\cdots,j_p\}\in T}}
\frac{w(e)}{k^p}
    \end{align}
\end{definition}

\paragraph{Structural Dictionary:}The tensors $\mathsf T^{(k,p)}$ define a hierarchy of tail co-occurrence
moments associated with order-$k$ directed hyperedges. The first few tail
moments recover familiar structural quantities:

\begin{align}
\mathsf T^{(k,0)}
&= d_\mathrm{in}^{(k)},
&& \text{(in-degree vector, see~\ref{def:in-degree})}\\
\mathsf T^{(k,1)}_{ij}
&= A_{ij}^{(k)},
&& \text{(exposure matrix/$k$-th order CHOM, see~\ref{eq:chol-matrix-decomposition})}\\
\mathsf T^{(k,2)}_{ij\ell}
&= \mathsf{P}_{ij\ell}^{(k)},
&& \text{(pair packing tensor, see~\ref{eq:pair-n-triple-packing-tensors})}\\
\mathsf T^{(k,3)}_{ij\ell m}
&= \mathsf{T}_{ij\ell m}^{(k)},
&& \text{(triple packing tensor, see~\ref{eq:pair-n-triple-packing-tensors})}
\end{align}
More generally,   $\mathsf T^{(k,p)}$ encodes  $p$-fold tail co-occurrence
statistics for order-$k$ hyperedges pointing to node $v_i$. Note that only the off-diagonal (i.e. distinct tail indices) elements capture genuine $p$-way interactions.

The tail moments are related to contractions of the adjacency tensor as follows

\begin{proposition}[$p$-th tail moments as  tensor contraction  ]\label{prop:pth-tail-moment-as-tensor-contraction}
 Let $\mathsf{A}^{(k)}$
denote the order-$(k+1)$ adjacency tensor of the $k$-uniform B-component $\mathcal{H}_k$. For an integer
$p\in\{0,1,\ldots,k\}$ and distinct tail indices $j_1,\dots,j_p$,  define the order-$(p+1)$ tensor
\begin{equation}
    \mathsf{A}^{(k,p)} = 
\frac{(k)_p}{k^p}
\left\langle
\mathsf{A}^{(k)},
\mathsf{1}_{k-p}
\right\rangle_{\{p+2,\dots,k+1\}},
\end{equation}

where $(k)_p=k(k-1)\cdots(k-p+1)$ is the falling factorial and
$\mathsf{1}_{k-p}$ is the all-ones tensor of order $k-p$.

Then the $p$-th tail moment tensor satisfies
\begin{equation}\label{eq:mth-order-tail-marginal-of-adjacency-tensor}
    \mathsf{T}^{(k,p)}_{i j_1\cdots j_p} = \mathsf{A}^{(k,p)}_{i j_1\cdots j_p}.
\end{equation}

Equivalently, letting $\mathbf 1\in\mathbb R^N$ denote the all-ones vector
and $\times_n$ the mode-$n$ product,
\begin{equation}
    \mathsf{T}^{(k,p)} =\frac{(k)_p}{k^p}
\,
\mathsf{A}^{(k)}
\times_{p+2}\mathbf 1^\top
\times_{p+3}\mathbf 1^\top \times
\cdots
\times_{k+1}\mathbf 1^\top.
\end{equation}

\end{proposition}
\begin{proof}
    See Appendix~\ref{app:der-prop-p-tail-moments-as-tensor-contraction}
\end{proof}

\paragraph{Structural Visibility:} One of the main take away of the the present work is establishing that \emph{different asymptotic orders of the dynamical reduction reveal different tail moments}. Figure~\ref{fig:structural-visibility} summarizes this principle.

\vspace{0.7em}

\begin{figure}[H]
    \centering
    \begin{tikzpicture}[node distance=1.1cm, every node/.style={font=\small}]
    \node[draw, rounded corners, fill=blue!20, inner sep=5pt] (A)
    {$\mathsf{A}^{(k)}\sim \mathsf{T}^{(k,k)}$};
    
    \node[draw, rounded corners, fill=green!20,
          below left=0.9cm and 1.7cm of A, inner sep=5pt] (T1)
    {$\mathsf{T}^{(k,1)}$};
    
    \node[draw, rounded corners, fill=yellow!20,
          below=0.9cm of A, inner sep=5pt] (T2)
    {$\mathsf{T}^{(k,2)}$};
    
    \node[draw, rounded corners, fill=orange!20,
          below right=0.9cm and 1.7cm of A, inner sep=5pt] (T3)
    {$\mathsf{T}^{(k,3)}$};
    
    \node[draw, dashed, rounded corners, fill=gray!10,
          right=of T3, xshift=1cm,
          inner sep=5pt] (Tp)
    {$\mathsf{T}^{(k,p)}$};

    \draw[-{Latex}] (A) -- (T1);
    \draw[-{Latex}] (A) -- (T2);
    \draw[-{Latex}] (A) -- (T3);
    \draw[gray,dashed,-{Latex}] (A) to[bend left=10] (Tp);
    \node[below=0.35cm of T1] {exposure};
    \node[below=0.65cm of T1] {(\emph{linear order})};
    \node[below=0.25cm of T2] {pair packing};
    \node[below=0.65cm of T2] {(\emph{quadratic order})};
    \node[below=0.25cm of T3] {triple packing};
    \node[below=0.65cm of T3] {(\emph{cubic order})};
    \node[below=0.25cm of Tp] {higher moments};
    \node[below=0.65cm of Tp] {(\emph{higher order})};
    \end{tikzpicture}
    \caption[Visibility hierarchy]{Visibility hierarchy structure illustrating the principle by which \emph{different asymptotic orders of the dynamical reduction reveal different tails moments}. This is one of the main take away of the present work. We show that at \emph{linear order}, the only structural information visible to the dynamics is \emph{exposure}, and everything else is discarded. Higher-order interactions only become relevant in the nonlinear regime. More specifically,  at \emph{quadratic order}, the dynamics sources \emph{pair packing} structural information encoded in the tail moment tensor $\mathsf{T}^{(k,2)}$, then contracts it with the critical directions to build the \emph{second} order coefficient $\beta$ of the reduced dynamics (see~\eqref{eq:amplitude-coefficients} ). At \emph{cubic order}, the dynamics  probes \emph{triple packing} structural information encoded in the tensor $\mathsf{T}^{(k,3)}$ in addition to pair packing information in $\mathsf{T}^{(k,2)}$ to build the \emph{cubic} order coefficient $\gamma$ through contraction with the critical directions as well. The process can be continued accordingly.}
    \label{fig:structural-visibility}
\end{figure}
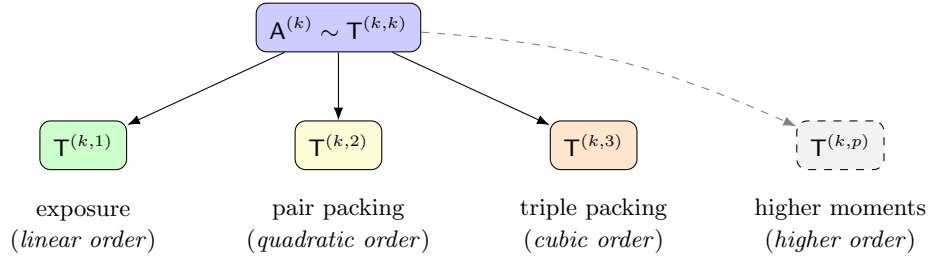

\subsection{Linearization about a reference trajectory}

We linearize  \eqref{eq:reaction-diffusion-on-hypergraph} about a reference trajectory $\bar{\mathbf{x}}(t) = \left (\bar{x}_1(t),\cdots,\bar{x}_N(t)\right )$, this leads to the linear dynamics describing the evolution of small perturbations around $\bar{\mathbf{x}}(t)$ as follows:
\begin{align}\label{Linearized_eq_nonlinear_coupling}
    \begin{split}
        \dot{\delta\mathbf{x}} &= J(\bar{\mathbf{x}})\delta\mathbf{x}
    \end{split}
\end{align}
where 
\begin{align}
    J(\bar{\mathbf{x}}) = \big( J_{i,\lambda}(\bar{\mathbf{x}})\big)_{i,\lambda}
\end{align}
is the $Nd \times Nd$ block-matrix whose  $d \times d$ block  entry $J_{i,\lambda}$ represents the  \emph{sensitivity} of the right hand side ($RHS(i)$) of \eqref{eq:reaction-diffusion-on-hypergraph} with respect to the state $x_{\lambda}$. More specifically, for any $1\le \lambda \le N$ we have 
\begin{align}
    J_{i,\lambda}(\bar{\mathbf{x}}) = \frac{\partial RHS(i)}{\partial x_{\lambda}} \Big|_{\! (\bar{\mathbf{x}})} 
\end{align}
where,
\begin{align}\label{eq:rhs-sensitivity}
    \begin{split}
        \frac{\partial RHS(i)}{\partial x_{\lambda}} \Big|_{\! (\bar{\mathbf{x}})}  &=  \frac{\partial f}{\partial x} (\bar{x}_i)\frac{\partial x_i}{x_{\lambda}} -  \sum_{k=1}^{\rho-1}{ \sum_{j_1,\cdots, j_k} \mathsf{L}^{(k)}_{ij_1\cdots j_k}}M_k\sum_{l=1}^k{\frac{\partial \psi_k}{\partial x_l}}(\bar{x}_{j_1},\cdots,\bar{x}_{j_k})\frac{\partial x_{j_l}}{x_{\lambda}} \\
        &= \frac{\partial f}{\partial x} (\bar{x}_i)\delta_{i\lambda} - \sum_{k=1}^{\rho-1}\sum_{l=1}^{k}\sum_{j_l=1}^{N} \left
        (\sum_{\hat{\jmath}_l}{\mathsf{L}^{(k)}_{ij_1\cdots j_k} M_k\frac{\partial \psi_k}{\partial x_l}(\bar{x}_{j_1},\cdots,\bar{x}_{j_k})\delta_{j_l\lambda} }    \right) 
    \end{split}
\end{align}
The summation over $\hat{\jmath}_l$ means that we are summing over all the indices $j_r$ \emph{except} $j_l$, and $\delta_{ij}$ is the usual Kronecker delta.

We now constrain the reference trajectory $\bar{\mathbf{x}}(t)$  to lie in the \emph{synchronization manifold} $\mathcal{S}$ \eqref{eq:sync-manifold}, thus \eqref{eq:rhs-sensitivity} becomes 

\begin{align}\label{eq:rhs-sensitivity-on-sync-manifold}
    \begin{split}
        \frac{\partial RHS(i)}{\partial x_{\lambda}} \Big|_{\! (\bar{\mathbf{x}})}  &=   \frac{\partial f}{\partial x} (\bar{x})\delta_{i\lambda} -\sum_{k=1}^{\rho-1}\sum_{l=1}^{k}\sum_{j_l=1}^{N} \mathsf{L}^{(k),\hat{l}}_{ij_l}M_k\frac{\partial \psi_k}{\partial x_l}(\bar{x},\cdots,\bar{x})\delta_{j_l\lambda}  
    \end{split}
\end{align}

where $\mathsf{L}^{(k),\hat{l}} $ is the $k$-th order $\hat{l}$-CHOM of  the Laplacian tensor defined in section~\ref{subsec:laplacian-tensor-formulation}. In fact using Lemma~\ref{lem:symmetry-and-contraction-invariance}, one can show that this operator is \emph{independent of $l$}. So, we refer to it as the $k$-th order  CHOL (Canonical Higher Order Laplacian).


The $\lambda$-th column of the block Jacobian matrix $J(\bar{\mathbf{x}})$ on $\mathcal{S}$ therefore reads:
\begin{align}\label{eq:column-of-Jacobian-matrix}
         \mathbf{e}_{\lambda} \otimes \frac{\partial f}{\partial x}(\bar{x}) - \sum_{k = 1}^{\rho-1}\sum_{l=1}^{k}\left( \textsf{L}^{(k),\hat{l}} \otimes M_k \right) \left(\mathbf{e}_{\lambda} \otimes \frac{\partial \psi_k}{\partial x_l}(\bar{x},\cdots,\bar{x})\right)
\end{align}

where $\left\{ \mathbf{e} \right\}_{1\le i\le N}$ is the canonical basis of $\mathbb{R}^N$. Putting everything together, we get the final form of the Jacobian matrix $J(\bar{\mathbf{x}})$ as follows:

\begin{align}
        J(\bar{\mathbf{x}}) &= I_N \otimes \frac{\partial f}{\partial x} (\bar{x}) - \sum_{k = 1}^{\rho-1}\sum_{l=1}^{k}\left( \mathsf{L}^{(k),\hat{l}} \otimes M_k \right) \left(I_N \otimes \frac{\partial \psi_k}{\partial x_l} (\bar{x},\cdots,\bar{x})\right) \notag\\
        &= I_N \otimes \frac{\partial f}{\partial x} (\bar{x})  - \sum_{k = 1}^{\rho-1}\left({ L^{(k)} \otimes M_k\sum_{l=1}^{k}\frac{\partial \psi_k}{\partial x_l} (\bar{x},\cdots,\bar{x}) }\right)  \label{eq:jacobian-derivation-nonlinear-diffusion}
\end{align}

where by Lemma~\ref{lem:symmetry-and-contraction-invariance},  the \(k\)-th order CHOL matrix \(L^{(k)}\) satisfies

\begin{equation}
    L^{(k)} = \mathsf{L}^{(k),\hat{l}^*} for \emph{ any } l^* \in \{1,\dots,k\}
\end{equation}

 It is worth pointing out that this operator is \emph{not symmetric} in general. In the undirected setting however, it is closely related to the \emph{Laplacian of the clique-expansion} of the $k$-uniform component of the underlying hypergraph. We will not develop this connection further here, and leave it for future communications.
 
 From tensor contraction of the  decomposition $\mathsf{L}^{(k)} = \mathsf{D}^{(k)} - \mathsf{A}^{(k)}$, together with   tail index symmetry in Lemma~\ref{lem:symmetry-and-contraction-invariance} we get
 \begin{align}\label{eq:chol-matrix-decomposition}
     L^{(k)} = D_\mathrm{in}^{(k)} - A^{(k)} \qquad 1 \le k \le N
 \end{align}

where $D_\mathrm{in}^{(k)}$ is the diagonal matrix of node in-degrees  at interaction order $k$, and $A^{(k)}$ is the $k$-th order CHOM of the adjacency tensor $\mathsf{A}^{(k)}$. Therefore, the vector $d_\mathrm{in}^{(k)} = \mathrm{diag}\left(D_\mathrm{in}^{(k)}\right)$ and the matrix $A^{(k)}$ can be viewed respectively as the \emph{zeroth} and  \emph{first tail moments}  of the tensor $\mathsf{A}^{(k)}$ by Definition~\ref{def:pth-tail-moment}. 


 \subsection{Working assumptions}

Although the linearization derived above applies to general multiway coupling functions in \eqref{eq:diffusive-coupling}, we now focus attention on a regime in which the \emph{mechanism governing pattern onset and saturation} can be clearly isolated.

\subsubsection{Order-independent coupling}\label{subsubsec:assumption-order-independent-coupling}
From now onward, we fix the aggregator in the class \eqref{eq:deepSet-aggregator-model} defined in section~\ref{sec:model-specification}, and we assume that it is independent of the interaction order $k$. This is equivalent to fixing the same activation function $\sigma$ across all the interaction layers. In addition, we also fix the mixing matrix $M$. Order dependence   enters only through the directed hypergraph structure and its relevant tensors.

\subsubsection{Regularity}\label{subsubsec:assumption-regularity}

The reaction term $f$ and activation function $\sigma$ are assumed to be  sufficiently smooth to admit a Taylor expansion about the homogeneous equilibrium $\bar{\mathbf{x}}  \in \mathcal{S}$ up to  third order at least, as required for weakly nonlinear analysis.

\subsubsection{Codimension-1 steady bifurcation}\label{subsubsec:codimention-1-steady-bifurcation}

For clarity of the presentation, we  focus only  on  \emph{steady instability} in the present work. More specifically, we assume there exists a critical parameter value $p_c$ (e.g. inhibitor diffusion coefficient) at which the linearized operator $J(\bar{\mathbf{x}})$ has a simple eigenvalue at zero, while all remaining eigenvalues have strictly negative real parts. Equivalently,
\begin{align}\label{eq:steady-instability-onset-assumption}
    \dim \ker J(\bar{\mathbf{x}}) = 1 =  \dim \ker(J(\bar{\mathbf{x}}))^\dagger 
\end{align}

where $(J(\bar{\mathbf{x}}))^\dagger$ denotes the adjoint of the operator $J(\bar{\mathbf{x}})$ with respect to the standard Euclidean inner product.

\emph{Oscillatory (Hopf) instabilities} and higher-codimension bifurcations are not considered here and are left for future work.

\subsection{Linearization theorem}

Under the nonlinear mean-aggregator $\eqref{eq:deepSet-aggregator-model}$, the Jacobian $\eqref{eq:jacobian-derivation-nonlinear-diffusion}$ simplifies drastically. In particular, the matrix multiplying each operator $L^{(k)}$ is independent of the interaction order $k$. As a result, the Jacobian depends on the hypergraph only through its zeroth and first tail moments.

\begin{theorem}[First tail-moment reduction]\label{thm:linearization-theorem}

Let $\bar{\mathbf{x}}$ be a homogeneous equilibrium of  \eqref{eq:hypergraph-dynamics} under nonlinear mean coupling \eqref{eq:deepSet-coupling-function}. Then the linearized dynamics about $\bar{\mathbf{x}}$ reads
\begin{align}
    \dot{\delta \mathbf{x}} &= \left( I_N \otimes J_f - \sum_{k=1}^{\rho-1} L^{(k)} \otimes D_{\mathrm{eff}} \right) \delta\mathbf{x}
\end{align}
or equivalently,
\begin{align}\label{eq:linearization-full}
    \dot{\delta \mathbf{x}} &= \big( I_N \otimes J_f - L_{\mathrm{tot}} \otimes D_{\mathrm{eff}} \big) \delta \mathbf{x}
\end{align}

where \(J_f = Df(\bar{x})\),  \( D_{\mathrm{eff}} = MD\sigma(\bar{x})\) and \( L_{\mathrm{tot}} = \sum_{k=1}^{\rho-1} L^{(k)}\).

In particular, the hypergraph influences the Jacobian only through its zeroth and  first tail moments, namely  the exposure information.
\end{theorem}

\begin{proof}
It follows from \eqref{eq:jacobian-derivation-nonlinear-diffusion} under the working assumptions \eqref{subsubsec:assumption-order-independent-coupling} and \eqref{subsubsec:assumption-regularity}.
\end{proof}

The theorem above shows that linear stability and Turing onset depend only on the zeroth and first tail moments of the adjacency tensor. The linearized operator is fully determined by these moments. Higher-order tail moments are invisible at linear level, and they reappear only through nonlinear terms. This separation underlies the mechanism developed Section~\ref{sec:nonlinear-saturation}.

\subsection{Directed dispersion relation}

We start with the following Lemma~\ref{lem:spectrum-of-full-Jacobian} which characterizes the spectrum of the Jacobian in \eqref{eq:linearization-full}.

\begin{lemma}\label{lem:spectrum-of-full-Jacobian}
    Let $\bar{\mathbf{x}} \in \mathcal{S}$ be a homogeneous equilibrium of  \eqref{eq:hypergraph-dynamics} under nonlinear mean coupling \eqref{eq:deepSet-coupling-function}, $J_f = Df(\bar{x})$, $L_{\mathrm{tot}} \in \mathbb{R}^{N \times N}$ and $D_{\mathrm{eff}} \in \mathbb{R}^{d \times d}$. The spectrum $\sigma(J)$ of the jacobian $ J = I_N \otimes J_f - L_{\mathrm{tot}} \otimes D_{\mathrm{eff}} $ of the system  is given by

    \begin{align}\label{eq:spectrum-of-Jacobian}
        \sigma(J) &= \bigcup_{\lambda \in \sigma \left(L_{\mathrm{tot}} \right)} \sigma \left(J_f - \lambda\, D_{\mathrm{eff}} \right)
    \end{align}
\end{lemma}

\begin{proof}
    The proof relies on the Schur decomposition (see, e.g., \cite{horn1945matrix,golub2013matrix,trefethen1997numerical}) as follows: Let $U^*L_{\mathrm{tot}} U = T$ for some unitary matrix $U$ and upper triangular matrix $T$. Then we have $(U^* \otimes I_d)\,L_{\mathrm{tot}} \,(U \otimes I_d) = I_N \otimes J_f - T \otimes D_{\mathrm{eff}}$. The latter matrix is block upper triangular with diagonal blocks $J_f - \lambda_i D_{\mathrm{eff}}$, where the $\lambda_i$'s are the eigenvalues of $L_{\mathrm{tot}}$ along the diagonal of $T$.
\end{proof}

\begin{definition}[Generalized dispersion relation]
\label{def:generalized-dispersion-relation}

For each
$
\lambda \in \mathbb{C},
$
define the linear operator $
A(\lambda) := J_f-\lambda D_{\mathrm{eff}}
\in \mathbb{R}^{d\times d}.
$
The associated \emph{generalized dispersion relation} is the function
\begin{equation}
\omega(\lambda) := \max\left\{ \Re(\nu)\,:\,
\nu\in\sigma\big(A(\lambda)\big) \right\}
\end{equation}
\end{definition}

By Lemma~\ref{lem:spectrum-of-full-Jacobian}, the spectrum of the Jacobian matrix 
$J = I_N\otimes J_f - L_{\mathrm{tot}}\otimes D_{\mathrm{eff}} 
$
is obtained by evaluating the spectra
$ \sigma(A(\lambda)) $
over all
$ \lambda\in\sigma(L_{\mathrm{tot}}). $
Thus, the homogeneous equilibrium is linearly stable if and only if
$
\omega(\lambda)<0,
\quad
\forall \lambda\in\sigma(L_{\mathrm{tot}}).
$
Instability onset occurs when
\begin{equation}
\max_{\lambda\in\sigma(L_{\mathrm{tot}})}
\omega(\lambda) = 0.
\end{equation}

In the undirected case,
$ L_{\mathrm{tot}}$ has a real spectrum, and
$\omega(\lambda)$
reduces to the classical real-valued dispersion relation. In contrast, directed interaction structures may yield a complex-valued spectrum, and this will require the dispersion relation to be evaluated over the complex plane. Moreover, since
$ L_{\mathrm{tot}} $
may also  be non-normal, transient amplification can occur even when
$\omega(\lambda)<0 $ for all
$ \lambda\in\sigma(L_{\mathrm{tot}})$.

\subsection{Linear indistinguishability}

Since the linearized operator in Theorem~\ref{thm:linearization-theorem} depends only on exposure structural information, hypergraphs that agree on their first tail moments are indistinguishable at linear level. We formalize this observation below.

\begin{definition}[Exposure-equivalence]\label{def:first-tail-moment-equivalence}

Two directed hypergraphs $\mathcal H_A$ and $\mathcal H_B$ are said to be exposure-equivalent or first-tail-moment equivalent if for every interaction order $k$, their first tail moments are equal:
\begin{align}
    A^{(k)}(\mathcal H_A)
    =
    A^{(k)}(\mathcal H_B),  \quad \text{ for all } k .
\end{align}

\end{definition}

Since $d_{in}^{(k)} = A^{(k)}\mathbf{1}$, then exposure-equivalence also implies equality of the zeroth tail moments.

\begin{corollary}[Linear indistinguishability]
\label{cor:linear-indistinguishability}

If two directed hypergraphs $\mathcal H_A$ and $\mathcal H_B$ are exposure-equivalent, then they induce identical linear operators. In particular, they have identical spectra, dispersion relations, and Turing onset thresholds.
\end{corollary}

\subsection{Interpretation}

Theorem~\ref{thm:linearization-theorem} reveals a strict hierarchy in the structural information encoded by the hypergraph. At linear level, the dynamics depends only on the first tail moment, and all higher-order tail moments are suppressed. Hypergraphs that coincide at this level are therefore indistinguishable with respect to stability and Turing onset.
The distinctions between such hypergraphs emerge only beyond onset, where nonlinear interactions reintroduce higher-order tail moments into the dynamics. Understanding how this mechanism occurs is the focus of the next section.

\section{Nonlinear Analysis and Pattern Saturation}\label{sec:nonlinear-saturation}

In this section, we develop a weakly nonlinear theory describing pattern formation near a steady (Turing) instability. Our goal is to identify the structural quantities that determine pattern saturation beyond the instability threshold. We show that while instability onset is fully determined by the first tail moment (exposure) tensor, the quadratic and cubic normal-form coefficients governing nonlinear saturation depend on higher-order tail moment (packing) tensors  that are completely invisible to linear theory. This reveals a deep hierarchy of structural information and establishes a clear separation between linear and nonlinear structural mechanisms (see Figure~\ref{fig:structural-visibility}). More specifically, we establish that exposure statistics suffice to determine the onset of instability, but packing statistics are required to determine post-onset pattern saturation.

\subsection{Spectral setup}\label{sub:spectral-setup}

Throughout this section, we assume  \eqref{subsubsec:codimention-1-steady-bifurcation}, namely a codimension-one steady instability at a control parameter $p$ (e.g. inhibitor diffusion coefficient), and set $p = p_c$ at criticality. It will be more convenient in subsequent analysis to use the detuning parameter $\epsilon = p - p_c$,  which measures the distance from onset.
Let us denote by $\dot{\mathbf{x}} = F(\mathbf{x};p_c + \epsilon)$ the dynamical system given by \eqref{eq:hypergraph-dynamics} under nonlinear mean coupling \eqref{eq:deepSet-coupling-function}, and let $\delta\mathbf{x} = \mathbf{x} - \bar{\mathbf{x}}$ denote the perturbation from the homogeneous equilibrium $\bar{\mathbf{x}} \in \mathcal{S}$. The third-order Taylor expansion  \eqref{subsubsec:assumption-regularity} of the vector field at $\bar{\mathbf{x}}$ reads
\begin{align}\label{eq:Taylor-expansion-of-full-vector-field}
    F(\bar{\mathbf{x}} + \delta\mathbf{x} ;p_c +\epsilon) &= J(\epsilon)\delta\mathbf{x} + B(\delta\mathbf{x},\delta\mathbf{x}) + C(\delta\mathbf{x},\delta\mathbf{x},\delta\mathbf{x}) + O(\| \delta\mathbf{x}\|^4)
\end{align}
where,  $J(\epsilon) = D_\mathbf{x}F(\bar{\mathbf{x}};p_c + \epsilon)$ is the Jacobian, and  
\begin{align}\label{eq:bilinear-n-trilinear-maps}
     B =  \frac{1}{2}D^2_\mathbf{x}F(\bar{\mathbf{x}};p_c), \quad\quad ~C = \frac{1}{6}D^3_\mathbf{x}F(\bar{\mathbf{x}};p_c)
\end{align}
are the scaled symmetric bilinear and trilinear maps corresponding to the second and third derivatives of the vector field $F$ at onset.

Since $F$ depends smoothly  on $\epsilon$, the Jacobian admits the following expansion
\begin{align}
    J(\epsilon) = J_0 + \epsilon J_1 + O(\epsilon^2)
\end{align}
where,  $J_0 = I_N \otimes J_f - L_{\mathrm{tot}} \otimes D_{\mathrm{eff}}$ is the Jacobian in the linearization \eqref{eq:linearization-full} at onset ($\epsilon = 0$), and
\begin{align}\label{eq:J1}
    J_1 &= \frac{\partial J(\epsilon)}{\partial\epsilon} \Big\rvert_{\epsilon = 0}
\end{align}

\begin{remark}
    Analogous expansions also hold for the maps $B$ and $C$. However, their $\epsilon$-dependence contributes terms of order $O(\epsilon\| \delta\mathbf{x}\|^2)$ and higher which  are absorbed into the remainder in the weakly nonlinear analysis considered below.
\end{remark}
The codimension-one assumption implies that $J_0$ has a simple eigenvalue at $0$ while all the remaining eigenvalues have strictly negative real part.

Using Lemma~\ref{lem:spectrum-of-full-Jacobian}, there exists a critical eigenvalue $\lambda_c$ of $L_{\mathrm{tot}}$ with right and left normalized eigenvectors $r,l \in \mathbb{R}^N$  satisfying
\begin{align}\label{eq:right-n-left-eigenvectors-of-Ltot}
    L_{\mathrm{tot}}\,r = \lambda_c\,r,  \quad\quad \l^{\top}L_{\mathrm{tot}} = \lambda_c\,l^{\top}, \quad\quad l^{\top}r = 1
\end{align}

Let $\xi,\eta \in \mathbb{R}^d$  denote  the normalized right and left  eigenvectors associated with the zero eigenvalue of the linear operator $J_f - \lambda_cD_{\mathrm{eff}}$,  so that
\begin{align}\label{eq:right-n-left-eigenvectors-of-Species-matrix}
    \left(J_f - \lambda_cD_{\mathrm{eff}} \right)\,\xi = 0,  \qquad \eta^{\top}\left(J_f - \lambda_cD_{\mathrm{eff}} \right) = 0, \quad\quad \eta^{\top}\xi = 1
\end{align}

Using the Kronecker product  structure of $J_0$, its  critical  right and left eigenvectors associated with the zero eigenvalue can be written in the following form
\begin{align}\label{eq:critial-right-n-left-eigenvectors-of-full-Jacobian}
    v = r\otimes \xi, \quad\quad w = l \otimes \eta
\end{align}
These critical  eigenvectors are normalized since $w^{\top}v = (l^{\top}r)(\eta^{\top}\xi)  = 1$.


\subsection{Weakly nonlinear reduction}

Standard weakly nonlinear theory \cite{carr2012applications,guckenheimer2013nonlinear,kuznetsov1998elements,wiggins2003introduction} implies that near the codimension-one steady instability described above, the dynamics reduces to a one-dimensional center manifold tangent to the right eigenvector $v$. The evolution on this manifold is governed by a scalar amplitude equation describing the growth and nonlinear saturation of the critical mode. The following proposition states this reduction,  and it provides generic expressions for the resulting coefficients, A detailed derivation is given in Appendix~\ref{app:der-prop-center-manifold-reduction}.

\begin{proposition}[Center manifold reduction and amplitude equation]\label{prop:center-manifold-reduction-n-amplitude-eqn}
    Under the above assumption \eqref{subsubsec:codimention-1-steady-bifurcation}, there exists a one-dimensional, locally invariant center manifold tangent to $\operatorname{span}\{v\}$ at $\delta\mathbf{x} = 0$, on which the system dynamics can be parametrized by a scalar amplitude $A(t)$ whose dynamics satisfies the following amplitude equation 
    \begin{align}\label{eq:amplitude-equation}
        \dot{A} = \epsilon\alpha A + \beta A^2 + \gamma A^3 + O(A^4)
    \end{align}
    where,
    \begin{align}\label{eq:amplitude-coefficients}
        \alpha = w^{\top}J_1v, \qquad \beta = w^{\top}B(v,v), \qquad \gamma = w^{\top}\Big(2B(v,u) + C(v,v,v) \Big)
    \end{align}
    and $u$ is the unique solution of the equation $J_0u = -QB(v,v)$, where $Q = I - vw^{\top}$ is the complementary projection onto the stable subspace.
\end{proposition}
\begin{proof}
    The result follows from standard center manifold reduction for a codimension-one steady bifurcation. A detailed derivation is provided in Appendix~\ref{app:der-prop-center-manifold-reduction}
\end{proof}

\subsection{Linear growth rate coefficient}

We now specialize the linear growth rate coefficient $\alpha$ appearing in the amplitude equation \eqref{eq:amplitude-equation} to our hypergraph setting. By leveraging the Kronecker product  structure of the linearized operator, we derive an explicit expression for $\alpha$ in terms of the critical eigenvalue $\lambda_c$ of $L_\mathrm{tot}$ and the "species" eigenvectors $p$ and $q$. The resulting formula reveals that the linear growth rate depends solely on the first-tail moment (exposure) information of the hypergraph, and it does not involve higher-order tails moments (packing) information. This is consistent with the linear analysis findings from the preceding section.

Let's recall from \eqref{eq:amplitude-coefficients} that $\alpha = w^\top J_1 v$, where $J_1$ is given by \eqref{eq:J1}. Equivalently, using the relation $p = p_c + \epsilon$ and the fact that the control parameter $p$ enters the Jacobian only through $D_\mathrm{eff}$, the chain rule yields 
\begin{align}
    J_1 &= -L_\mathrm{tot}\otimes\frac{\partial D_\mathrm{eff}}{\partial p }\Big\rvert_{p=p_c}
\end{align}

Altogether, using \eqref{eq:right-n-left-eigenvectors-of-Ltot} and \eqref{eq:critial-right-n-left-eigenvectors-of-full-Jacobian} we get
\begin{align}
    \alpha &= -\left (l^\top \otimes \eta^\top\right ) L_\mathrm{tot}\otimes\frac{\partial D_\mathrm{eff}}{\partial p } \left (r\otimes \xi\right ) \notag\\
    &= -\lambda_c\, \eta^\top \left (\frac{\partial D_\mathrm{eff}}{\partial p } \right ) \xi \label{eq:alpha-with-Deff}
\end{align}
where the derivative is evaluated at the critical parameter value $p=p_c$.

For clarity of presentation, we specialize expression \eqref{eq:alpha-with-Deff} to the classical two-species setting $d=2$. In this case, the state variable is $x=(u,v)$, where $u$ and $v$ represent the activator and inhibitor species respectively. The resulting formula for the growth-rate coefficient $\alpha$ is stated in the following proposition.

\begin{proposition}[Linear growth rate coefficient in the two-species case]
Consider the hypergraph reaction–diffusion system \eqref{eq:hypergraph-dynamics} with nonlinear mean coupling \eqref{eq:deepSet-coupling-function}. Assume $d=2$, corresponding to the classical two-species setting $x=(u,v)$, where $u$ and $v$ denote the activator and inhibitor species  respectively. Let the effective diffusion matrix be $D_\mathrm{eff} = M D\sigma(\bar{x})$, 
as defined in \eqref{eq:linearization-full}, where $M=\operatorname{diag}(d_u,d_v)$ and $d_u,d_v$ are the activator and inhibitor diffusion coefficients. Using the notation introduced in Section~\ref{sub:spectral-setup}, the linear growth-rate coefficient in the amplitude equation \eqref{eq:amplitude-equation} is given by
\begin{align}
    \alpha = -\lambda_c \, \sigma_v'(\bar{v}) \, \eta_v\,\xi_v
\end{align}

where $\lambda_c$ is the critical eigenvalue of the operator $L_\mathrm{tot}$, and \; $\xi_v$, $\eta_v$ denote the $v$-components of the corresponding right and left species eigenvectors.
\end{proposition}
\begin{proof}
    It follows from  \eqref{eq:alpha-with-Deff} by taking  $p=d_v$ as the control parameter, and using the fact that the activation $\sigma$ acts componentwise.
\end{proof}

The expression for $\alpha$ shows that the onset growth rate depends on the hypergraph structure only through the critical eigenvalue $\lambda_c$ of the  operator $L_\mathrm{tot}$ (i.e.  \emph{exposure Laplacian}). The remaining factors are determined by the local kinetics and the activation function $\sigma$. In particular, no higher-order tail moment (\emph{packing}) information appears. Therefore,  exposure-equivalent hypergraphs exhibit identical linear growth rates. This mirrors the \emph{Linear distinguishability} result~\ref{cor:linear-indistinguishability}  established in the linear analysis section.

\subsection{Nonlinear saturation coefficients}

To understand how the hypergraph structure influences pattern saturation, we leverage the natural decomposition of the vector field $F$ into its reaction and diffusion components. Under nonlinear mean coupling \eqref{eq:deepSet-coupling-function}, the diffusion term can be  further separated into contributions from the head node and contributions from multiway tail interactions. This separation induces a corresponding split of the quadratic and cubic multilinear maps $B$ and $C$ \eqref{eq:bilinear-n-trilinear-maps}, and hence of the amplitude equation coefficients $\beta$ and $\gamma$ given in  \eqref{eq:amplitude-coefficients}. 

Using the equivalent hyperedge formulation of the dynamics stated in Proposition~\ref{prop:hyperedge-formulation}, the multilinear operators decompose as
\begin{align}\label{eq:multilinear-maps-decomposition}
    B = B^\mathrm{react} + B^\mathrm{head} + B^\mathrm{tail}, \qquad C = C^\mathrm{react} + C^\mathrm{head} + C^\mathrm{tail}
\end{align}
where, for arbitrary perturbation $\delta\mathbf{x}, \delta\mathbf{y}$ and $\delta\mathbf{z}$ of the homegeneous equilibrium $\bar{\mathbf{x}} \in \mathcal{S}$, and $i \in \{1,\cdots,  N\}$, we have
\begin{align}
B^\mathrm{react}_i(\delta\mathbf{x},\delta\mathbf{y}) &= \frac{1}{2}D^2f(\bar{x})(\delta x_i,\delta y_i) \label{eq:B-react} \\
B^\mathrm{head}_i(\delta\mathbf{x},\delta\mathbf{y}) &= -\frac{1}{2} \left( \sum_{k=1}^{\rho-1}{ d_\mathrm{in}^{(k)}(v_i)} \right) MD^2\sigma(\bar{x})(\delta x_i, \delta y_i) \label{eq:B-head}
\end{align}
and the following tail piece from which the second-order tail moment (\emph{pair packing}) tensor emerges
\begin{align}\label{eq:B-i-tail}
B^\mathrm{tail}_i(\delta\mathbf{x},\delta\mathbf{y}) &= \frac{1}{2}\sum_{k=1}^{\rho-1}\sum_{e=(T\rightarrow i)\in E_k}\sum_{j,l \in T} { \frac{w(e)}{k^2} M D^2 \sigma(\bar{x})(\delta x_j,\delta y_l)}
\end{align}
Similarly, one can derive the expressions for $C^\mathrm{react}$ and $C^\mathrm{head}$, with the latter involving only the zeroth tail moment tensor $d_\mathrm{in}^{(k)}$ as in \eqref{eq:B-head}. For the sake of completeness, we give below the expression for $C^\mathrm{tail}$ since as in \eqref{eq:B-i-tail}, it reveals the origin of the third tail moment (\emph{triple packing}) tensor. 
\begin{align}\label{eq:C-i-tail}
C^\mathrm{tail}_i(\delta\mathbf{x},\delta\mathbf{y},\delta\mathbf{z}) &= \frac{1}{6}\sum_{k=1}^{\rho-1}\sum_{e=(T\rightarrow i)\in E_k}\sum_{j,l,t \in T} { \frac{w(e)}{k^3} M D^3 \sigma(\bar{x})(\delta x_j,\delta y_l,\delta z_t)}
\end{align}

\subsubsection{Pair and triple packing tensors}
The expressions for $B^\mathrm{tail}$ and $C^\mathrm{tail}$ above show how the higher-order tail moments arise naturally from the dynamics. Using Definition~\ref{def:pth-tail-moment}, we now introduce the \emph{pair} and \emph{triple} packing tensors, which correspond to the second and third tail-moment tensors, respectively. 

\begin{align}\label{eq:pair-n-triple-packing-tensors}
\mathsf{P}^{(k)}_{i j_1 j_2}
\coloneqq
\sum_{\substack{e=(T\rightarrow i)\in E_k, \\ j_1,j_2\in T}}
\frac{w(e)}{k^2},
\qquad \qquad
\mathsf{Q}^{(k)}_{i j_1 j_2 j_3}
\coloneqq
\sum_{\substack{e=(T\rightarrow i)\in E_k, \\ j_1,j_2,j_3\in T}}
\frac{w(e)}{k^3}
\end{align}

A few observations are in order here. First, the tensor $\mathsf{P}^{(k)}$ measures pairwise co-occurrence of nodes within B-hyperedges of tail size $k$ pointing to a given head node  in $\mathcal{H}_k$, while $\mathsf{Q}^{(k)}$ measures triple co-occurrence. Next, these quantities capture hyperedge "packing" patterns, and they contain structural information that \emph{cannot} be inferred from the first tail moment (exposure) tensor alone when $k \ge 2$. Finally, Proposition~\ref{prop:pth-tail-moment-as-tensor-contraction} shows that these \emph{packing tensors} arise as  \emph{contractions} of the \emph{adjacency tensor}. More specifically, we have
\begin{align}\label{eq:packing-tensors-as-contraction}
    \mathsf{P}^{(k)}_{ij_1j_2} = \mathsf{A}^{(k,2)}, \qquad \qquad \mathsf{Q}^{(k)}_{ij_1j_2j_3} = \mathsf{A}^{(k,3)}.
\end{align}


The structure  of the packing tensors is determined by $k$. At $k=1$, tails are singletons, so all nonzero entries occur when indices are equal. Both tensors reduce to diagonal tensors with entries given by the exposure matrix $A^{(1)}$~(see \eqref{eq:chol-matrix-decomposition}) and therefore carry no nontrivial co-occurrence information. More specifically, we have
\begin{align}
    \mathsf{P}_{i j_1 j_2}^{(1)} = \delta_{j_1 j_2}\,A_{i j_1}^{(1)}, \qquad \qquad
\mathsf{Q}_{i j_1 j_2 j_3}^{(1)} = \delta_{j_1 j_2} \delta_{j_2 j_3}\,A_{i j_1}^{(1)}.
\end{align}
For $k \ge 2$, $\mathsf{P}^{(k)}$ has support on distinct index pairs and is genuinely pair-packing-driven. Meanwhile, The tensor $\mathsf{Q}^{(2)}$ has no support on triples of distinct indices and reduces to pairwise co-occurrence structure via repeated indices. Genuine triple packing appears only for $k \ge 3$, where $\mathsf{Q}^{(k)}$ admits nonzero entries on triples of distinct indices and is not reducible to lower-order statistics.

\paragraph{Summary of Packing Types}

\begin{align}
    \mathsf{P}^{(k)} =
\begin{cases}
\text{exposure}, & k=1, \\
\text{pair-packing}, & k \ge 2,
\end{cases}
\qquad \qquad
\mathsf{Q}^{(k)} =
\begin{cases}
\text{exposure}, & k=1, \\
\text{pair-packing}, & k=2, \\
\text{triple-packing}, & k \ge 3.
\end{cases}
\end{align}
In particular, \emph{genuine }$r$-way packing requires tails of size at least $r$, which is  reflected in the tensor contractions \eqref{eq:packing-tensors-as-contraction} and more generally in \eqref{eq:mth-order-tail-marginal-of-adjacency-tensor}.

\subsubsection{Decomposition of the saturation coefficients}
It follows from \eqref{eq:amplitude-coefficients} that the nonlinear coefficients $\beta$ and $\gamma$ are obtained by projecting the quadratic and cubic nonlinear forcing terms onto the critical left eigenvector $w$. As a result, the decomposition of the multilinear maps in \eqref{eq:multilinear-maps-decomposition} induces a corresponding decomposition of the nonlinear coefficients as follows.
\begin{align}\label{eq:decomposition-of-beta-and-gamma}
    \beta = \beta_\mathrm{react} + \beta_\mathrm{head} + \beta_\mathrm{tail}, \qquad \qquad
    \gamma = \gamma^B + \gamma^C_\mathrm{react} + \gamma^C_\mathrm{head} + \gamma^C_\mathrm{tail}
\end{align}
where, $ \beta_\mathrm{react}, \beta_\mathrm{head}$ and  $\gamma^C_\mathrm{react}, \gamma^C_\mathrm{head}$ are obtained by projecting the appropriate component of the  quadratic and cubic maps respectively onto the critical left eigenvector. The component  $\gamma^B$ is given by
\begin{align}\label{eq:gamma_B-component}
    \gamma^B = 2w^\top B(v,u)
\end{align}
The tail contribution to the quadratic saturation coefficient  $\beta$, and the tail contribution of the trilinear map $C$ to the cubic saturation coefficient $\gamma$ are given by 
\begin{align}
    \beta_\mathrm{tail} = w^\top B^\mathrm{tail}(v,v), \qquad \text{and } \qquad \gamma^C_\mathrm{tail} = w^\top C^\mathrm{tail}(v,v,v).
\end{align}

\subsubsection{How packing tensors enter the saturation coefficients}

We now explicitly show how  packing tensors contribute to the nonlinear saturation coefficients. We first consider the quadratic forcing  term. Substituting the critical right eigenvector $v =r\otimes \xi$ into \eqref{eq:B-i-tail}, we get 
\begin{align}
    B^\mathrm{tail}_i(v,v) &= \frac{1}{2}\left( \sum_{k=1}^{\rho-1}\Pi^{(k)}_i\right ) MD^2\sigma(\bar{x})(\xi,\xi) \notag\\
    &= \frac{1}{2}\left(\Pi_i^{(1)}+ \Pi^\mathsf{P}_i \right) MD^2\sigma(\bar{x})(\xi,\xi)
\end{align}
where the order-k pair-packing scalars obtained by contraction with the  right eigenvector $r$ of $L_\mathrm{tot}$, and their tail sum, are defined by
\begin{align}\label{eq:contracted-pair-packing-scalar}
\Pi^{(k)}_i =\sum_{j_1,j_2=1}^{N}\mathsf{P}^{(k)}_{ij_1j_2}r_{j_1}r_{j_2}, \qquad \text{and }\qquad \Pi^\mathsf{P}_i = \sum_{k=2}^{\rho-1}\Pi^{(k)}_i
\end{align}
Collecting these scalars into $N$-dimensional vectors $\Pi^{(k)}$ and $\Pi^\mathrm{tot}$, we get 
\begin{align}\label{eq:B-tail-kronecker-form}
    B^\mathrm{tail}(v,v) = \frac{1}{2}\left(\Pi^{(1)}+ \Pi^\mathsf{P} \right)\otimes MD^2\sigma(\bar{x})(\xi,\xi)  
\end{align}
where the order-k contracted pair-packing vector, and their tail sum, are defined by
\begin{align}\label{eq:contracted-pair-packing-vectors}
    \Pi^{(k)} \coloneqq \mathsf{P}^{(k)}\times_2r\times_3r, \qquad \text{ and }\qquad \Pi^\mathsf{P} \coloneqq \sum_{k=2}^{\rho-1}\Pi^{(k)}
\end{align}

Projecting \eqref{eq:B-tail-kronecker-form} onto the critical left eigenvector $w$, we get the following decomposition of the tail piece of the quadratic saturation coefficient into its \emph{diagonal} component associated with $\Pi^{(1)}$ and its \emph{pair-packing} component which is associated with $\Pi^\mathsf{P}$
\begin{align}\label{eq:tail-contribution-in-beta}
    \beta_\mathrm{tail} = \beta_\mathrm{diag} + \beta_\mathrm{P}
\end{align}

Altogether, the quadratic saturation coefficient decomposes as follows
\begin{align}\label{eq:decomposition-of-beta}
    \beta =\left( \beta_\mathrm{react} + \beta_\mathrm{head} + \beta_\mathrm{diag} \right) + \beta_\mathsf{P}
\end{align}

A similar argument shows that the cubic tail forcing term can be written in the following compact form
\begin{align}\label{}
    C^\mathrm{tail}(v,v,v) = \frac{1}{6} \left(\Gamma^{(1)}+ \Gamma^{(2)}+\Gamma^\mathsf{Q} \right)\otimes MD^3\sigma(\bar{x})(\xi,\xi,\xi)  
\end{align}
where the order-k contracted triple-packing vector, and  their tail sum, are defined by
\begin{align}\label{eq:contracted-triple-packing-vectors}
    \Gamma^{(k)} \coloneqq \mathsf{T}^{(k)}\times_2 r\times_3 r\times_4 r, \qquad \text{ and }\qquad \Gamma^\mathsf{Q} \coloneqq \sum_{k=3}^{\rho-1}\Gamma^{(k)}
\end{align}

The tail contribution of the trilinear map to the cubic saturation coefficient decomposes into its \emph{diagonal} component driven by $\Gamma^{(1)}$, its \emph{pair-packing} component driven by $\Gamma^{(2)}$ and its \emph{triple-packing} component driven by $\Gamma^{\mathsf{Q}}$
\begin{align}
    \gamma^C_\mathrm{tail} = \gamma^C_\mathrm{diag}+ \gamma^C_\mathsf{P} + \gamma_\mathsf{Q} 
\end{align}
Hence, we get the complete decomposition of the cubic saturation coefficient  as follows
\begin{align}\label{eq:decomposition-of-gamma}
    \gamma = \gamma^C_\mathrm{react} + \gamma^C_\mathrm{head} + \gamma^C_\mathrm{diag} + \gamma^B + \gamma^C_\mathsf{P} + \gamma_\mathsf{Q}
\end{align}

The following proposition summarizes the derivation above, and  it   gives explicit expressions for the nonlinear saturation coefficients in our hypergraph setting


\begin{proposition}[Nonlinear saturation coefficients]\label{prop:nonlinear-saturation-coefficients}
Consider the  reaction–diffusion system on hypergraph \eqref{eq:hypergraph-dynamics} with nonlinear mean coupling \eqref{eq:deepSet-coupling-function}.
    Let $v = r\otimes \xi$ and $w = l \otimes \eta$ denote the critical right and left eigenvectors at onset as defined in Section~\ref{sub:spectral-setup}. The nonlinear saturation coefficients in the amplitude equation \eqref{eq:amplitude-equation} are given by
\begin{align}\label{eq:decomposition-of-nonlinear-coefficients-inside-the-proposition}
    \begin{split}
        \beta &= \beta_\mathrm{react} + \beta_\mathrm{head} + \beta_\mathrm{diag} + \beta_\mathsf{P} \\
    \gamma &= \gamma^C_\mathrm{react} + \gamma^C_\mathrm{head} + \gamma^C_\mathrm{diag} + \gamma^B + \gamma^C_\mathsf{P} + \gamma_\mathsf{Q}
    \end{split}
\end{align}
where, the reaction and head components are defined as follows
\begin{align}
    \beta_\mathrm{react} = \frac{1}{2}\left(l^\top r^{\odot 2}\right) \left( \eta^\top f_2 \right), \qquad 
    \gamma^C_\mathrm{react} = \frac{1}{6}\left(l^\top r^{\odot 3}\right) \left( \eta^\top f_3 \right)
\end{align}
\begin{align}\label{eq:head-terms-beta-and-gamma}
    \begin{split}
        \beta_\mathrm{head} &= -\frac{1}{2}\left(l^\top\left( d_\mathrm{in}^\mathrm{tot} \odot r^{\odot 2} \right)\right) \left( \eta^\top s_2 \right) \\ 
    \gamma^C_\mathrm{head} &= -\frac{1}{6}\left(l^\top\left( d_\mathrm{in}^\mathrm{tot} \odot r^{\odot 3} \right)\right) \left( \eta^\top s_3 \right).
    \end{split}
\end{align}
The diagonal components of the saturation coefficients are given by
\begin{align}\label{eq:diagonal-components-to-beta-and-gammma}
    \beta_\mathrm{diag} = \frac{1}{2}\left(l^\top \Pi^{(1)}\right) \left( \eta^\top s_2 \right), \qquad 
    \gamma^C_\mathrm{diag} = \frac{1}{6}\left(l^\top \Pi^{(1)}\right) \left( \eta^\top s_3 \right).
\end{align}
The direct pair-packing  components to both coefficients are given by
    \begin{align}\label{eq:direct-pair-packing-componebts-to-beta-and-gamma}
        \beta_\mathsf{P} = \frac{1}{2}\left(l^\top\Pi^\mathsf{P} \right) \left(\eta^\top s_2 \right), \qquad
        \gamma^C_\mathsf{P} = \frac{1}{6}\left(l^\top\Gamma^{(2)} \right) \left(\eta^\top s_3 \right).
    \end{align}
The direct triple-packing component to the cubic saturation coefficient reads 
\begin{align}\label{eq:direct-triple-packing-componebts-to-gamma}
        \gamma_\mathsf{Q} = \frac{1}{6}\left(l^\top\Gamma^\mathsf{Q} \right) \left(\eta^\top s_3 \right).
    \end{align}

For convenience, we denote the  quadratic and cubic reaction responses along the critical mode by
    \[
        f_2 \coloneqq D^2f(\bar{x})(\xi,\xi), \qquad \text{ and } \qquad f_3 \coloneqq D^3f(\bar{x})(\xi,\xi,\xi).
    \]
    Similarly, the species-side responses of the activation function along the critical mode  is denoted by
     \[
        s_2 \coloneqq MD^2\sigma(\bar{x})(\xi,\xi), \qquad \text{ and } \qquad s_3 \coloneqq MD^3\sigma(\bar{x})(\xi,\xi,\xi).
    \]
We denote the $m$-th elementwise power of $r$ by $r^{\odot m}$, namely $\left(r^{\odot m} \right)_i = r_i^m$.
The total in-degree vector $d_\mathrm{in}^\mathrm{tot}$ is  defined as follows
\[
    d_\mathrm{in}^{(k)} \coloneqq \Big(d_\mathrm{in}^{(k)}(v_1),\cdots,d_\mathrm{in}^{(k)}(v_N) \Big)^\top, \qquad \text{and} \qquad d_\mathrm{in}^\mathrm{tot} \coloneqq \sum_{k=1}^{\rho-1}d_\mathrm{in}^{(k)}.
\]

Finally, $\Pi^\mathsf{P}$ and $\Gamma^\mathsf{Q}$ are the contracted intrinsic pair and triple-packing vectors defined in \eqref{eq:contracted-pair-packing-vectors} and \eqref{eq:contracted-triple-packing-vectors}.
\end{proposition}
\begin{proof}
The proof leverages the  Kronecker product structure of the equations, and it follows from \eqref{eq:decomposition-of-beta-and-gamma} -- \eqref{eq:decomposition-of-gamma}. 
\end{proof}

\begin{remark}
The cubic saturation coefficient $\gamma$ also contains the indirect  contribution $\gamma^B = 2 w^\top B(v,u),$
where $u$ is the unique solution to  the  equation $J_0u = -QB(v,v)$, and the operator  $Q = I-vw^\top$ is the projection onto the stable subspace. The term $\gamma^B$ inherits pair-packing dependence through the "feedback" term $u$, but it does not reduce to a direct contraction formula of the same form as $\beta_\mathsf{P}, \gamma^C_\mathsf{P}$ and $\gamma_\mathsf{Q}$.
\end{remark}

The key implication of Proposition~ \ref{prop:nonlinear-saturation-coefficients} is that the zeroth tail moment (exposure) information contributes to both saturation coefficients through their head and diagonal terms \eqref{eq:head-terms-beta-and-gamma}--\eqref{eq:diagonal-components-to-beta-and-gammma}, whereas higher-order tail moment (packing) information enter through the tail terms. In particular, only pair packing contributes to the quadratic coefficient $\beta$ via its component $\beta_\mathsf{P}$. Meanwhile the cubic coefficient $\gamma$ receives  pair-packing effect through $\gamma^B$ and $\gamma^C_\mathsf{P}$, and triple packing effect through $\gamma_\mathsf{Q}$. Therefore, \emph{exposure} information alone is \emph{insufficient} to determine pattern saturation, \emph{packing} information is \emph{required}.

\subsection{Normal form regimes}

The amplitude equation derived in \eqref{eq:amplitude-equation} includes both quadratic and cubic nonlinear terms, leading to two distinct normal form regimes based on whether the quadratic coefficient $\beta$ vanishes. In the \emph{resonant case} ($\beta \neq 0$), the quadratic term governs the leading nonlinear dynamics, whereas in the \emph{non-resonant case} ($\beta = 0$),  the cubic term determines the leading nonlinear saturation, thus leading to the classic Stuart-Landau equation.

In contrast to classical reaction–diffusion systems, where symmetry often forces quadratic terms to cancel out \cite{newell1969finite,cross1993pattern,hoyle2006pattern,cross2009pattern}, no such cancellation mechanism exists  a priori in hypergraph-coupled settings. This reflects a structural feature of our  framework, namely nonlinear mean aggregation produces intrinsic quadratic contributions that generically persist under projection onto the critical mode. As a result, the condition $\beta = 0$ is non-generic and would require specific parameter tuning to achieve. The simple example provided below illustrates this mechanism explicitly.

\begin{example}[Genericity of the resonant case]\label{exmpl:generecity-of-the-resonant-case}
    We consider three scalar dynamical units $\{v_1,v_2,v_3 \}$ interacting over the hypergraph whose hyperedges are $e_1 = \left(\left\{v_2,v_3 \right\} \rightarrow v_1 \right),\, e_2 = \left(v_1  \rightarrow v_2 \right) $ and $e_3 = \left(v_1  \rightarrow v_3 \right)$. The weight function $\omega$ assigns positive weights to each hyperedge as follows: $\omega(e_1) = \omega_h$ \,  and $\omega(e_2) =  \omega_e=\omega(e_3)$.  The interaction is mediated by a nonlinear mean-aggregation activation function $\sigma$ as in \eqref{eq:deepSet-coupling-function} and   satisfying $\sigma(0) = 0,~\sigma'(0) > 0$ and $\sigma''(0)\ne 0$. We set $M = -d$, where $d>0$ is the control parameter. Each unit evolves according to the same intrinsic dynamics $f(x) = c_1x+c_2x^2$, where $c_1 < -\omega_0d\sigma'(0)$. In fact, any sufficiently smooth $f$ satisfying  $f(0) =0$ with appropriately bounded first derivative will do. The dynamics of the resulting higher-order network system reads
    \begin{align}\label{eq:toy-example-dynamics}
        \begin{split}
            \dot{x}_1 &= f(x_1)-\omega_h d\left( \sigma\left( \frac{x_2+x_3}{2}\right)-\sigma(x_1)  \right)\\
        \dot{x}_2 &= f(x_2) -\omega_ed\left( \sigma(x_1)-\sigma(x_2)\right)\\
        \dot{x}_3 &= f(x_3) - \omega_ed\left( \sigma(x_1)-\sigma(x_3)\right)
        \end{split}
    \end{align}
The homogeneous equilibrium is $\bar{\mathbf{x}} = (0,0,0)^\top$, and the Jacobian matrix  $J = I_3\otimes J_f - L_\mathrm{tot}\otimes D_\mathrm{eff}(d)$ has eigenvalues $\lambda_1 = c_1+(\omega_e+\omega_h)d\sigma'(0)$, \, $\lambda_2 = c_1+ \omega_ed\sigma'(0)$ and $\lambda_3 = c_1$. So, a codimension-$1$ steady bifurcation occurs at
\begin{align}
    d_c = -\frac{c_1}{(\omega_e+\omega_h)\sigma'(0)}
\end{align}
At onset ($d = d_c$), the normalized critical right and left eigenvectors are given by
\begin{align}\label{eq:critical-eigenventor-in-the-example}
    v = (\omega_h,-\omega_e,-\omega_e)^\top, \qquad \text{ and } \qquad w = \frac{1}{\omega_e+\omega_h}\left(1,-\frac{1}{2}, -\frac{1}{2} \right)^\top
\end{align}
A Taylor expansion of the vector field in \eqref{eq:toy-example-dynamics} yields the following quadratic forcing vector
\begin{align}
    B(v,v) =
\begin{pmatrix}
     c_2\omega_h^2+\frac{1}{2}(\omega_h^2-\omega_e^2)\omega_h d_c\sigma''(0) \\ c_2\omega_e^2+\frac{1}{2}(\omega_e^2-\omega_h^2)\omega_e d_c\sigma''(0)\\ c_2\omega_e^2+\frac{1}{2}(\omega_e^2-\omega_h^2)\omega_e d_c\sigma''(0) 
\end{pmatrix}
\end{align}    
The quadratic nonlinear saturation coefficient is obtained by projection onto the critical left eigenvector  $w$, and it reads
\begin{align}\label{eq:toy-example-beta}
    \beta = (\omega_h-\omega_e) \left( c_2 - \frac{ \sigma''(0)}{2\sigma'(0)}c_1 \right)
\end{align}  
Thus, if we assume $\omega_e \ne \omega_h$,  the condition $\beta = 0$ requires   fine-tuning   the system parameters to lie exactly on the hypersurface defined by setting the right-hand side of \eqref{eq:toy-example-beta} to zero, and it is therefore non-generic.
\end{example}

Since the condition $\beta = 0$ is non-generic, the weakly nonlinear dynamics falls into two normal form regimes depending on whether the quadratic saturation coefficient vanishes. We will make use of the time rescaling $T = \epsilon t$ to derive the resulting \emph{slow-time} dynamics in both cases.

\subsubsection{Resonant case $\beta \ne 0$}
The amplitude equation \eqref{eq:amplitude-equation} retains its form $
        \dot{A} = \epsilon\alpha A + \beta A^2 + \gamma A^3 + O(A^4)
$.

Balancing the linear and quadratic terms yields $A = O(\epsilon)$. Thus  introducing the amplitude rescaling
$
    A = \epsilon a
$, 
we arrive at the \emph{slow-time} dynamics
\begin{align}
    a_T = \alpha a + \beta a^2 + \epsilon \gamma a^3 + O(a^4)
\end{align}
 which to leading order, gives the following \emph{transcritical-type} normal form
 \begin{align}\label{eq:transcritical-normal-form}
     a_T = \alpha a + \beta a^2
 \end{align}

In this regime, the quadratic coefficient $\beta$ determines the orientation of the bifurcating branch, while the linear growth rate coefficient $\alpha$ controls its stability. This reflects the exchange-of-stability mechanism characteristic of transcritical bifurcations.

When the nontrivial branch is stable, namely $\alpha\epsilon > 0$, the amplitude saturates at
\begin{align}\label{eq:saturating-amplitude-resonant-case}
    A_\infty(\epsilon) =  - \frac{\alpha}{\beta}\epsilon + O(\epsilon^2)
\end{align}

\subsubsection{Non-resonant case $\beta = 0$}
Here, the amplitude equation \eqref{eq:amplitude-equation} takes the form $
        \dot{A} = \epsilon\alpha A  + \gamma A^3 + O(A^4)
$.

Balancing the linear and cubic terms yields $A = O(\epsilon^\frac{1}{2})$ so that the amplitude rescaling $A = \epsilon^\frac{1}{2} a$ gives the following \emph{slow-time} dynamics
\begin{align}
    a_T = \alpha a + \gamma a^3 + O(a^4)
\end{align}
To leading order, we get the following \emph{pitchfork-type} normal form characteristic of the steady Stuart-Landau equation
\begin{align}\label{eq:transcritical-normal-form}
     a_T = \alpha a + \gamma a^3
 \end{align}

In this regime, the linear growth rate coefficient $\alpha$ determines the side of onset on which the homogeneous state loses stability, while the cubic coefficient $\gamma$ controls the criticality of the bifurcation. More specifically, the bifurcation is supercritical for $\gamma <0$, and it is subcritical for $\gamma >0$.

When the nontrivial branch is stable, which occurs whenever $\gamma < 0$ and $\alpha \epsilon >0 $ hold simultaneously, the amplitude saturates at
\begin{align}\label{eq:saturating-amplitude-non-resonant-case}
    A_\infty(\epsilon) =  \pm \sqrt{-\frac{\alpha}{\gamma}\epsilon }+ O(\epsilon)
\end{align}

\begin{remark}
    In both the resonant and non-resonant regimes, packing terms enter the leading nonlinear saturation coefficient. Although these terms may vanish in special cases, such as when the critical mode does not probe the tail directions as in Example~\ref{expl:vanishing-of-packing-contribution} below, they are generically present. Therefore, outside such degenerate configurations, the post-onset dynamics is intrinsically shaped by packing effects, thus  emphasizing the robustness of the mechanism.
\end{remark}

\subsection{ Nonlinear distinguishability and structural effects}\label{subsec:structural-non-identifiability}

Up to this point, we have  established that packing statistics is the structural information controlling the nonlinear saturation of the dynamics. In this section, we take a step further to highlight a core subtlety of this mechanism. In particular, we show that structural presence alone is not sufficient for packing tensors to be  \emph{dynamically relevant}, even though these tensors naturally arise whenever non-trivial hyperedges are present. Rather, packing must be "\emph{seen by the critical mode}", in the sense that its \emph{contribution} 
to the relevant saturation coefficients is nonzero.

\begin{example}[Vanishing of  quadratic packing effect $\beta_\mathsf{P}$]\label{expl:vanishing-of-packing-contribution}
    Beyond providing a template for the \emph{genericity of the resonant case} in hypergraph setting, Example~\ref{exmpl:generecity-of-the-resonant-case} reveals an important subtlety in the  system dynamics. It follows from \eqref{eq:toy-example-beta} that
    \begin{align}\label{eq:beta-in-the-remark-vanishing-pair-packing-contribution}
            \begin{split}
                \beta &= \beta_\mathrm{react}+ \beta_\mathrm{head}+ \beta_\mathrm{diag} + \beta_\mathsf{P} \\
        &= c_2(\omega_h-\omega_e) -\kappa(\omega_h^3-\omega_e^3) - \kappa\,\omega_e\omega_h^2 + \kappa\,\omega_h\omega_e^2
            \end{split}
    \end{align}
    where, 
    \begin{align}
        \kappa\coloneqq \frac{1}{(\omega_e+\omega_h)^2} \frac{ \sigma''(0)}{2\sigma'(0)}c_1
    \end{align}
    This matches the decomposition  predicted by \eqref{eq:decomposition-of-nonlinear-coefficients-inside-the-proposition}. 
    
    Now,  observe that if we set $\omega_e = 0$, meaning removing all pairwise interactions, then the tail contribution \eqref{eq:tail-contribution-in-beta} disappears from \eqref{eq:beta-in-the-remark-vanishing-pair-packing-contribution}. In particular, the quadratic packing effect vanishes (i.e. $\beta_\mathsf{P} = 0$), despite the persistence of the nontrivial hyperedge $e_1$. This should not be interpreted as the system being intrinsically pairwise,  rather it indicates that  higher-order interactions are dynamically inactive at quadratic order. This reflects the structure of the critical mode $v$ (see \eqref{eq:critical-eigenventor-in-the-example}), which in this case has no support on the tail nodes when $\omega_e=0$.
    \end{example}

This example highlights a key  insight about the mechanism governing post-onset dynamics. The dynamical relevance of packing tensors is not determined solely by the presence of higher-order hyperedges, but by how they  interact with the critical mode. In particular, packing tensors enter the reduced dynamics (i.e. amplitude equation) only through their contraction with the critical eigendirections, as seen in \eqref{eq:contracted-pair-packing-vectors}  and \eqref{eq:contracted-triple-packing-vectors}. Consequently, \emph{higher-order interactions may be structurally present, yet dynamically invisible at leading order}.

Packing tensors encode node co-occurrence within hyperedges in a combinatorial sense, but only the portion that projects onto the critical eigendirections influences the reduced dynamics. We call this portion \emph{packing effect}, as it represents the part of packing that actually enters the saturation coefficients in the amplitude equation, and it should be distinguished from packing as a purely structural feature (i.e packing tensor). At quadratic order, the packing effect is simply 
\begin{equation}\label{eq:beta-pack}
    \beta_{\mathrm{pack}} = \beta_\mathsf{P}
\end{equation}
while at cubic order it takes the form 
\begin{equation}\label{eq:gamma-pack}
    \gamma_{\mathrm{pack}} = \gamma^B_\mathrm{pack} +  \gamma^C_\mathsf{P} + \gamma_\mathsf{Q}
\end{equation}
where the terms represent direct second and third-order contributions together with indirect (feedback) effects arising through the center-manifold correction $\gamma^B$.

We can now state the following important theorem.

\begin{theorem}[Nonlinear structural decomposition]\label{thm:nonlinear-structural-decomposition}
    The amplitude saturation coefficients decompose as
\begin{equation}
    \beta=\beta_\mathrm{exp}+\beta_\mathrm{pack},
\quad \text{ and } \quad
\gamma=\gamma_\mathrm{exp}+\gamma_\mathrm{pack}
\end{equation}
where the exposure-driven parts $\beta_{\mathrm{exp}}$ and $\gamma_{\mathrm{exp}}$ collect contributions  from the reaction, head, and diagonal components, while the packing-driven parts $\beta_{\mathrm{pack}}$ and $\gamma_{\mathrm{pack}}$ are simply the packing effects given in  \eqref{eq:beta-pack} and \eqref{eq:gamma-pack}
\end{theorem}
\begin{proof}
    I follows from \eqref{eq:decomposition-of-beta-and-gamma} through \eqref{eq:decomposition-of-gamma}.
\end{proof}

This motivates the following definition of weak packing-equivalence

\begin{definition}[Weak packing-equivalence]\label{def:weak-packing-equivalence}
Consider two hypergraph reaction–diffusion systems $\mathcal{H}_A$ and $\mathcal{H}_B$ with dynamics  \emph{(\eqref{eq:hypergraph-dynamics} + \eqref{eq:deepSet-coupling-function})} under the same local reaction and activation functions, and admitting reduced amplitude equations of the form
\[
\dot{A} = \mu A + \beta A^2 + \gamma A^3 + O(A^4).
\]

For $m\in\{2,3\}$, we say that the hypergraphs $\mathcal{H}_A$ and $\mathcal{H}_B$ are weakly packing-equivalent to order $m$ if their packing effects agree up to  order $m$. In particular, for $m = 2$, we have
\begin{equation}
    \beta_{\mathrm{pack}}(\mathcal{H}_A) = \beta_{\mathrm{pack}}(\mathcal{H}_B)
\end{equation}
while for $m = 3$ we have
\begin{equation}
    \beta_{\mathrm{pack}}(\mathcal{H}_A) = \beta_{\mathrm{pack}}(\mathcal{H}_B), \qquad \text{and} \qquad
    \gamma_{\mathrm{pack}}(\mathcal{H}_A) = \gamma_{\mathrm{pack}}(\mathcal{H}_B)
\end{equation}
\end{definition}

\begin{remark}
The qualifier "weak" in Definition~\ref{def:weak-packing-equivalence} indicates that only the packing effects appearing in the saturation coefficients are compared, rather than the full packing tensors or the underlying hypergraph structure.
\end{remark}

The following theorem characterizes the condition under which two hypergraph-coupled systems are indistinguishable at the level of the amplitude equation, thus identifying the structural features that govern this dynamics.

\begin{theorem}[Reduced dynamics equivalence]\label{thm:nonlinear-identifiability}
Fix $m\in \{2,3\}$, and consider two hypergraph reaction–diffusion systems $\mathcal{H}_A$ and $\mathcal{H}_B$ with dynamics  \emph{(\eqref{eq:hypergraph-dynamics} + \eqref{eq:deepSet-coupling-function})} under the same local reaction and activation functions, and admitting reduced amplitude equations of the form
\[
\dot{A} = \mu A + \beta A^2 + \gamma A^3 + O(A^4).
\]

If the two systems are exposure-equivalent (see Definition~\ref{def:first-tail-moment-equivalence}) and weakly packing-equivalent to order $m$ (see Definition~\ref{def:weak-packing-equivalence}), then their reduced dynamics agree up to order $m$. In particular, for $m = 2$, we have
\begin{equation}
    \mu(\mathcal{H}_A) = \mu(\mathcal{H}_B) \qquad \text{and} \qquad
    \beta(\mathcal{H}_A) = \beta(\mathcal{H}_B)
\end{equation}
while for $m = 3$ we have
\begin{equation}
    \mu(\mathcal{H}_A) = \mu(\mathcal{H}_B), \qquad
    \beta(\mathcal{H}_A) = \beta(\mathcal{H}_B), \qquad
    \text{and} \qquad
    \gamma(\mathcal{H}_A) = \gamma(\mathcal{H}_B)
\end{equation}
\end{theorem}
\begin{proof}
    First, note that the linear coefficient $\mu = \epsilon\alpha$ depends only on exposure information (see \eqref{eq:alpha-with-Deff}). Next, exposure-equivalence implies equality of $\mu$, $\beta_\mathrm{exp}$ and $\gamma_\mathrm{exp}$, while  weak packing-equivalence implies equality of $\beta_\mathrm{pack}$ and $\gamma_\mathrm{pack}$. Finally, conclude using Theorem~\ref{thm:nonlinear-structural-decomposition}.
\end{proof}

The following  corollary characterizes nonlinear distinguishability.

\begin{corollary}[Nonlinear distinguishability]\label{cor:nonlinear-distinguishability}
Fix an integer $m\in \{2,3\}$. If two hypergraph-coupled systems are exposure-equivalent, but not weakly packing-equivalent to order $m$, then their reduced dynamics must differ at some order $1< p \le m$. 
\end{corollary}
\begin{proof}
    The proof follows from Theorem~\ref{thm:nonlinear-identifiability}.
\end{proof}

The dependence of the reduced dynamics on exposure and packing effects leads naturally to the concept of \emph{dynamical graph surrogacy}, which identifies the conditions under which a hypergraph-coupled system can be represented by an equivalent graph at the level of its reduced dynamics.

\begin{corollary}[Dynamical graph surrogacy]\label{cor:graph-surrogacy}
    Fix $m \in \{2,3\}$. A hypergraph-coupled system admits a \emph{dynamical graph surrogate} up to order $m$ if and only if its packing effects vanish up to order $m$. In this case, the system is dynamically indistinguishable from a graph-based system at the level of the reduced dynamics up to order $m$.
\end{corollary}
\begin{proof}
The result follows from Theorem~\ref{thm:nonlinear-identifiability}, and noting that packing effects \eqref{eq:beta-pack}--\eqref{eq:gamma-pack} quantify the "obstruction" to representing hypergraph dynamics by an exposure-driven (i.e. graph) model at the level of the reduced dynamics. 
\end{proof}

It is important to emphasize that \emph{dynamical graph surrogacy} does not imply the absence of higher-order interactions in the hypergraph. Rather, it reflects the \emph{absence of their net contribution} to the reduced dynamics up to the specified order.

\begin{remark}[Quadratic surrogacy and curvature-free activation]\label{rem:curvature-free-activation-and-quad-surrogacy}

We note that the quadratic packing effect coefficient $\beta_\mathsf{P}$ depends jointly on the higher-order hypergraph structure and on the local geometry of the activation function (see \eqref{eq:direct-pair-packing-componebts-to-beta-and-gamma}). In particular, if the activation has vanishing second derivative (i.e. $D^2\sigma(\bar{x}) \coloneqq \sigma''(\bar{x})$, which we call curvature since the activation acts componentwise) at the homogeneous equilibrium, then the quadratic packing effect necessarily vanishes. More specifically
\begin{equation}\label{eq:zero-curvature}
    D^2\sigma(\bar{x}) = 0
    \qquad \Longrightarrow \qquad
    \beta_\mathsf{P}=0.
\end{equation}

Thus, even in the presence of nontrivial higher-order interactions, quadratic packing effects become dynamically invisible whenever the activation is locally curvature-free around equilibrium.

Typical examples include
\begin{equation}
    \sigma(x)=a+b\tanh(x-\bar{x}),
    \qquad
    \sigma(x)=x+(x-\bar{x})^{p},
    \quad p>2.
\end{equation}

For such activations, the weakly nonlinear reduction eliminates the quadratic packing term entirely.

This reveals two distinct mechanisms by which higher-order effects may fail to appear in the reduced dynamics. The first is \emph{structural invisibility}, arising when the relevant packing effects vanish because of the hypergraph organization itself. The second is \emph{functional invisibility}, arising when the local curvature of the interaction nonlinearity suppresses higher-order effects independently of the underlying topology.

Consequently, the dynamical visibility of higher-order interactions depends not only on combinatorial structure, but also on the local differential geometry of the nonlinear response.
\end{remark}

\subsection{Interpretation}

The nonlinear analysis shows that higher-order interactions affect the dynamics solely through their contributions to the saturation coefficients. These contributions are therefore the only features of higher-order structure that are reflected in the reduced dynamics, and differences in them lead to distinguishable nonlinear behavior.
This observation motivates the notion of dynamical graph surrogacy. A hypergraph-coupled system may be dynamically equivalent to a graph despite the presence of higher-order interactions, provided that their net contribution to the reduced dynamics vanishes.

\section{Numerical Validation}\label{sec:section6}

We now validate the theoretical predictions on finite hypergraph-coupled reaction-diffusion systems. Since exposure-equivalence guarantees identical linearized dynamics by construction, the numerical experiments focus primarily on the nonlinear regime, where differences in packing effects become dynamically observable.

The main objective is to test two theoretical predictions developed in the previous sections. More specifically, we ascertain that exposure-equivalence alone does not determine nonlinear behavior beyond instability onset, and dynamic graph surrogacy is obstructed by nonzero packing effects. This latter statement simply means that a graph surrogate reproduces the hypergraph dynamics only when the relevant packing effects vanish.

\subsection{Numerical setup}

We consider finite directed hypergraph-coupled reaction-diffusion systems with Schnakenberg \cite{schnakenberg1979simple}  local kinetics. Each node carries state variables $(u_i,v_i)$ satisfying
\begin{align}
    \dot{u}_i &= a-u_i+u_i^2v_i \\
    \dot{v}_i &= b-u_i^2v_i
\end{align}

Throughout the numerical experiments, the local kinetic parameters are fixed at $a = 0.2$ and $b=1.3$.
The corresponding homogeneous equilibrium is

\begin{equation}
u^*=a+b=1.5,
\qquad
v^*=\frac{b}{(a+b)^2}\approx0.57778.
\end{equation}

For all interaction orders, the mixing matrices are taken to be
$
    M=\mathrm{diag}(d_u,d_v),
$
where $d_u = 0.25$, and  the inhibitor diffusion coefficient $d_v$ serving as the bifurcation parameter. 

To generate nonlinear higher-order effects, we use polynomial activation functions of the form
\begin{equation}
    \sigma(z) = c_1 z+c_2 z^2+c_3 z^3.
\end{equation}

Separate activation functions are used for the activator and inhibitor components:
\begin{equation}
\sigma_u(z)
=
c_{1,u}z+c_{2,u}z^2+c_{3,u}z^3,
\qquad
\sigma_v(z)
=
c_{1,v}z+c_{2,v}z^2+c_{3,v}z^3.
\end{equation}

Unless otherwise specified, the coefficients are chosen as $
c_{1,u}=1.1,~
c_{2,u}=0.2,~
c_{3,u}=0.05,
$ and $
c_{1,v}=5.9,~
c_{2,v}=-0.15,~
c_{3,v}=0.4.
$
These choices generate nontrivial quadratic and cubic nonlinearities and allow controlled suppression of the pair packing effect in the graph-surrogacy experiments through the condition
\(
\sigma''(\bar{x})=0
\).

First, we consider the pair of hypergraphs whose complete description is given in Table~\ref{tab:HA_HB} in the Appendix. This pair is exposure-equivalent but not weakly packing-equivalent. This pair is  used to validate nonlinear distinguishability under exposure-equivalence shown in Figure~\ref{fig:pairI_distinguishability}. Then we use the hypergraph described in Table~\ref{tab:H3} for all subsequent experiments.


All simulations are performed near the instability threshold,
\begin{equation}
    d_v=d_v^c+\varepsilon,
\qquad
\varepsilon>0,
\end{equation}
where $d_v^c$ denotes the critical value predicted by the linear theory.

The modal amplitude is extracted by projection onto the critical left eigenvector $w$ as follows
\begin{equation}
    A_{\mathrm{num}}(t)
=
 w^\top\left(X(t)-\bar{X}\right),
\end{equation}
where $X(t)$ denotes the full system state, $\bar X$ is the homogeneous equilibrium, and the critical eigenvectors are normalized by
$
    w^\top v=1.
$

The saturated amplitude $A_{\mathrm{sat}}$ is estimated by averaging $|A_{\mathrm{num}}(t)|$ over the terminal portion of the trajectory after convergence. Additional diagnostics include the relative standard deviation and relative drift over the averaging window to ensure proper convergence.

\subsection{Nonlinear Distinguishability under Exposure-Equivalence}

Figure~\ref{fig:pairI_distinguishability} shows the saturated amplitudes near onset for Pair I.

Although the two systems possess identical exposure operators and therefore identical linear instability thresholds, their post-onset amplitudes differ substantially as one moves farther away from onset. This demonstrates that exposure-equivalence alone is insufficient to determine nonlinear dynamics beyond onset.

The observed discrepancy is consistent with the theory developed earlier. While the two systems agree at linear order, their packing effects differ, thus leading to different reduced coefficients in the amplitude equation and therefore different nonlinear saturation behavior.

Table~\ref{tab:pairI_coefficients} reports the corresponding coefficient data for Pair I, including the common instability threshold $d_v^c$, the packing effects $\beta_{\mathrm{pack}}$ and $\gamma_{\mathrm{pack}}$, and the measured saturation amplitudes. The table makes explicit that the two systems are linearly indistinguishable but nonlinearly distinguishable.

\begin{figure}[H]
\centering
\includegraphics[width=0.49\textwidth]{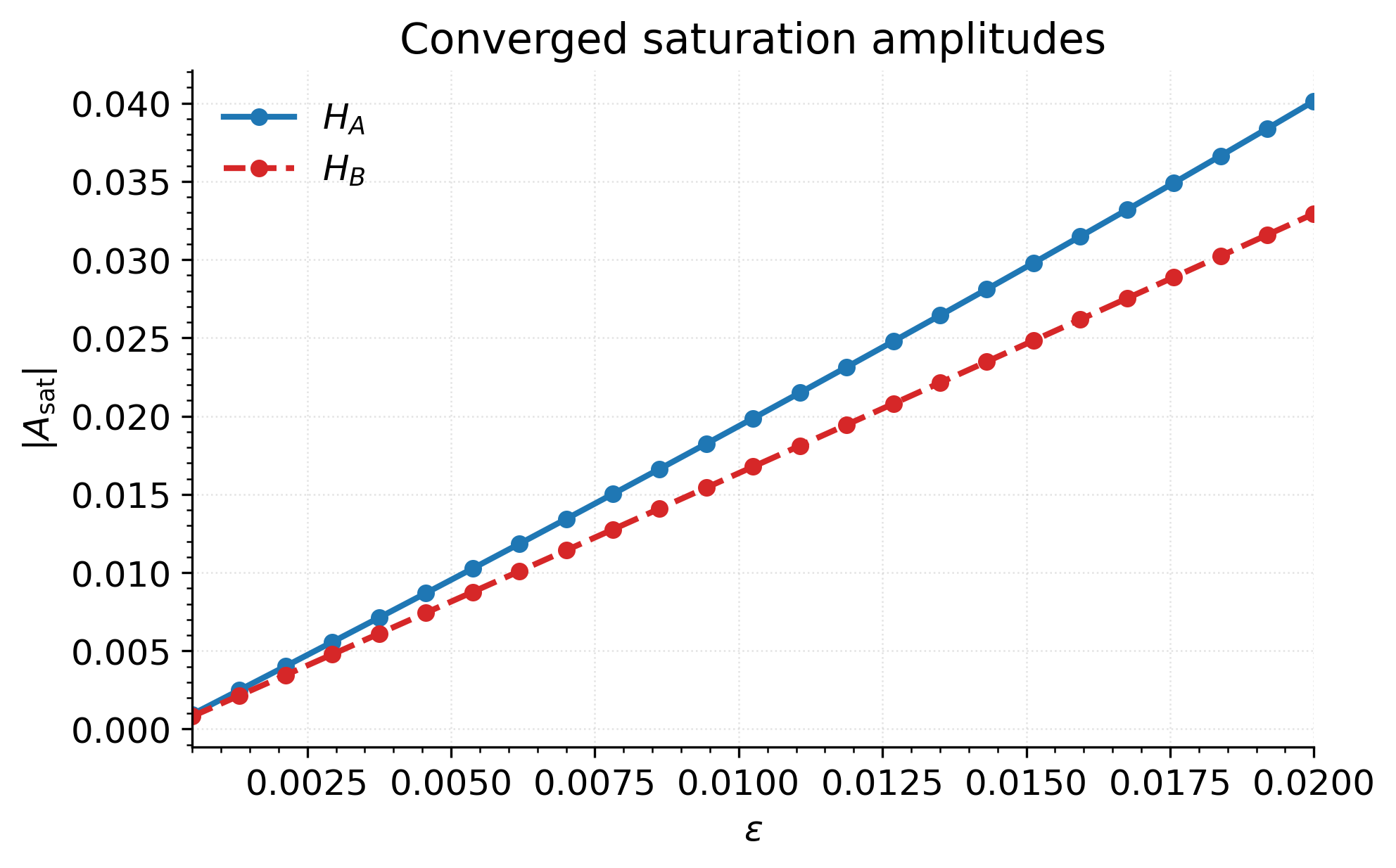}
\includegraphics[width=0.49\textwidth]{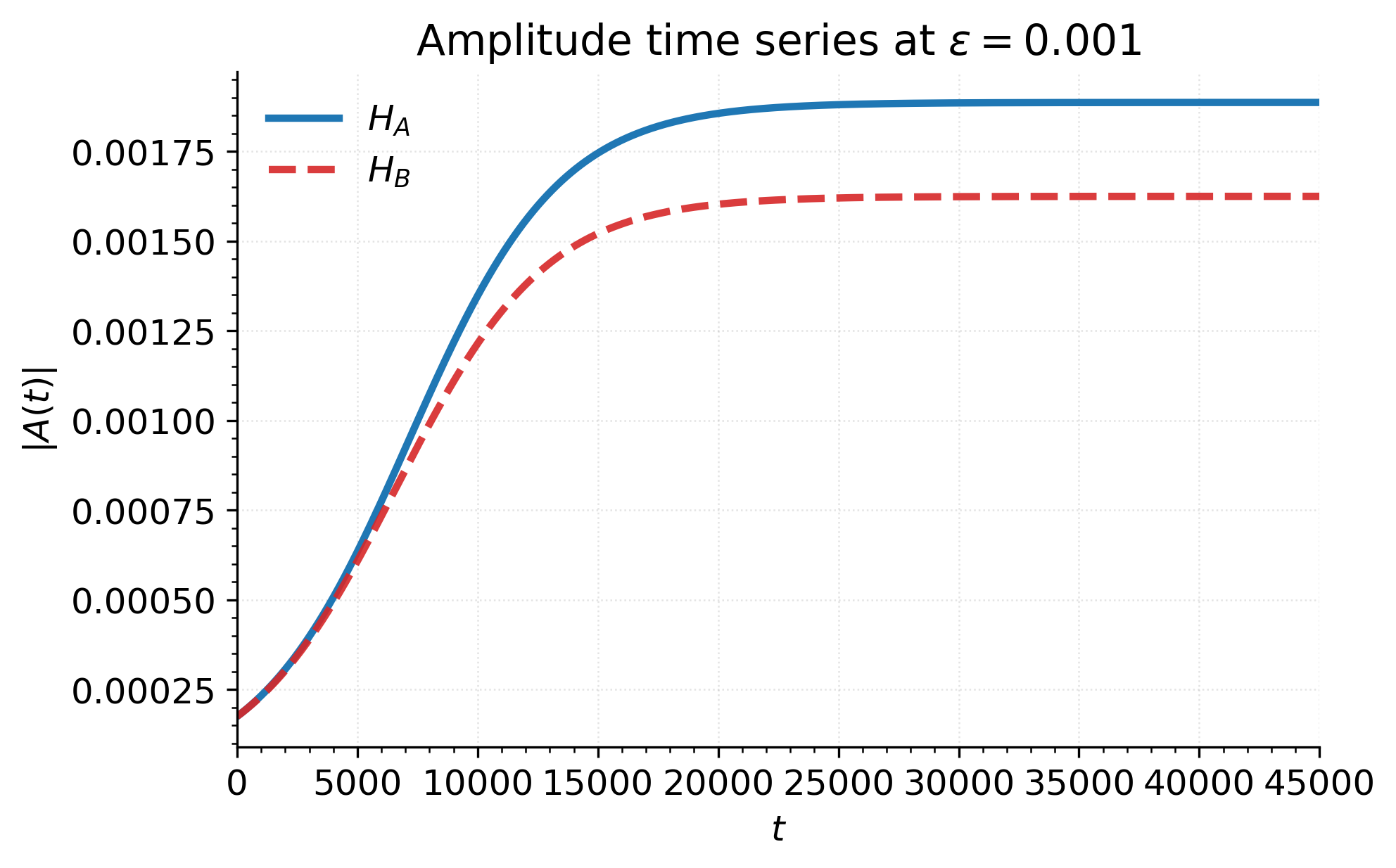}
\caption[Nonlinear distinguishability for Pair in Table~\ref{tab:HA_HB}]{
Nonlinear distinguishability for the pair of hypergraphs in Table~\ref{tab:HA_HB}.
\textbf{Left:}
Saturated amplitudes as a function of $\varepsilon = d_v - d_v^c$.
Although $\mathcal{H}_A$ and $\mathcal{H}_B$ are exposure-equivalent and therefore share an identical linear instability threshold, their nonlinear saturation amplitudes differ due to distinct packing effects. The discrepency is more pronounced the further away from onset (i.e larger $\varepsilon$). Moreover, the plots are consistent with the asymptotic prediction in \eqref{eq:saturating-amplitude-resonant-case}.
\textbf{Right:}
Representative amplitude trajectories $A(t)$ at fixed $\varepsilon = 0.001$.
Both systems exhibit similar early-time growth near onset consistent with exposure-equivalence, but  they eventually converge to distinct saturated states. This demonstrates nonlinear distinguishability beyond the linear regime.
}
\label{fig:pairI_distinguishability}
\end{figure}

\begin{table}[H]
\centering
\caption{
Coefficient comparison for the pair of hypergraphs in Table~\ref{tab:HA_HB}.
The two systems are exposure-equivalent and therefore share the same linear instability threshold $d_v^c$, but they differ in their packing effects and nonlinear reduced coefficients. The resulting  saturated amplitudes $|A_{\mathrm{sat}}|$ confirms nonlinear distinguishability beyond the linear regime. See Figure~\ref{fig:pairI_distinguishability} for the corresponding numerical comparison.
}
\label{tab:pairI_coefficients}
\renewcommand{\arraystretch}{1.15}
\begin{tabular}{c c c c c c c}
\toprule
System
& $d_v^c$
& $\beta$
& $\beta_{\mathrm{pack}}$
& $\gamma$
& $\gamma_{\mathrm{pack}}$
& $|A_{\mathrm{sat}}|$
\\

\midrule

$\mathcal{H}_A$
& 1.30079
& 0.17014
& -0.01313
& 0.15544
& 0.06949
& 0.00189
\\

$\mathcal{H}_B$
& 1.30079
& 0.19713
& 0.0
& 0.00603
& 0.0
& 0.00162
\\
\bottomrule
\end{tabular}
\end{table}

\subsection{Dynamical Graph Surrogacy}

We next test the graph-surrogacy prediction.

Given the hypergraph $\mathcal{H}$ described in Table~\ref{tab:H3}, we construct an exposure-preserving graph surrogate $\mathcal{G}$ by replacing each hyperedge
$
e \coloneqq \{j_1,\dots,j_k\} \rightarrow i
$
of weight $w(e)$ by the  edges
$
 j_m \rightarrow i 
$
of weight $\frac{w(e)}{k}$ for each $m=1,\dots,k$. This construction preserves the exposure operator exactly, and therefore preserves the linearized dynamics.

We compare the projected amplitudes
\begin{align}
    \begin{split}
        A_\mathcal{H}(t) &=
w^\top \left(X_\mathcal{H}(t)-\bar{X}\right) \\
A_\mathcal{G}(t) &=
w^\top \left(X_\mathcal{G}(t)-\bar{X} \right)
    \end{split}
\end{align}

Two activation regimes are considered.

\paragraph{Case 1: Vanishing quadratic packing effect}

When $\sigma''(\bar x)=0$, the quadratic packing effect vanishes (see Remark~\ref{rem:curvature-free-activation-and-quad-surrogacy}).  The graph surrogate reproduces the hypergraph dynamics to quadratic order near onset. Numerically, the trajectories remain close, and since  cubic packing effect is also small, the amplitudes nearly coincide. Figure \ref{fig:graph_surrogacy} (Left panel ) illustrates this behavior

\paragraph{Case 2: Nonzero quadratic packing effect}

When $\sigma''(\bar x)\neq0$, quadratic packing effects become active. In this regime, the exposure-preserving graph surrogate no longer reproduces the hypergraph dynamics. Although the two systems initially agree due to their identical linearization, they separate at nonlinear order and saturate at different amplitudes.

Figure~\ref{fig:graph_surrogacy} (Right panel) illustrates this behavior. The result demonstrates that packing effects provide an obstruction to graph surrogacy, in the sense that preserving exposure is sufficient to reproduce linear onset, but it is  insufficient to reproduce nonlinear hypergraph dynamics.

Table~\ref{tab:graph_surrogacy} summarizes the graph-surrogacy experiments, including the common instability threshold, the packing effects $\beta_{\mathrm{pack}}$ and $\gamma_{\mathrm{pack}}$, the final saturated amplitudes, and the saturation discrepancy
\begin{equation}
    \Delta A_{\mathrm{sat}}
=
|A_{\mathcal{H},\mathrm{sat}}-A_{\mathcal{G},\mathrm{sat}}|.
\end{equation}

\begin{figure}[H]
\centering
\includegraphics[width=1\textwidth]{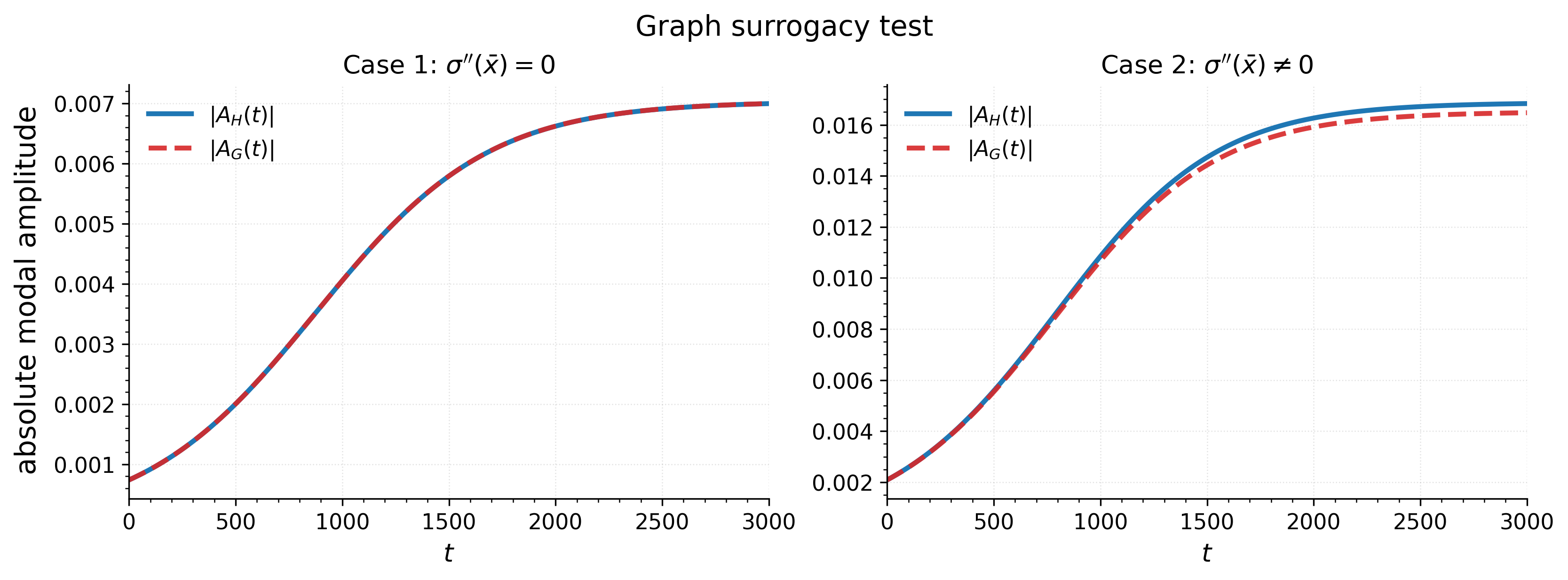}
\caption[Dynamical graph surrogacy test]{
Graph-surrogacy test for a hypergraph $\mathcal{H}$ and its exposure-preserving graph surrogate $\mathcal{G}$.
\textbf{Left:}
Activation regime satisfying
$\sigma''(\bar{x})=0$.
In this regime, the quadratic packing effect vanishes, suppressing the leading nonlinear obstruction to graph surrogacy. The projected amplitudes $A_\mathcal{H}(t)$ and $A_\mathcal{G}(t)$ remain very close throughout the evolution, with only small deviations attributable to higher-order packing effects.
\textbf{Right:}
Activation regime satisfying
$\sigma''(\bar{x})\neq 0$.
Quadratic packing effects are active, and the hypergraph and graph surrogate exhibit visibly different nonlinear dynamics despite identical linear onset. The systems initially evolve similarly but saturate at different amplitudes, demonstrating that exposure-equivalence alone is insufficient to guarantee nonlinear equivalence.
}
\label{fig:graph_surrogacy}
\end{figure}

\begin{table}[H]
\centering
\caption{
Summary of the graph-surrogacy experiments.
For each activation regime, the table reports the packing effects and the resulting saturated amplitudes for the hypergraph $\mathcal{H}$ and its exposure-preserving graph surrogate $\mathcal{G}$. Agreement is recovered when the relevant packing effects vanish, while nonzero packing effects produce observable nonlinear discrepancies despite identical linear dynamics. See Figure~\ref{fig:graph_surrogacy} for the corresponding visual comparison where we set $\varepsilon = 0.001$.
}
\label{tab:graph_surrogacy}
\renewcommand{\arraystretch}{1.15}
\begin{tabular}{c c c c c c c c c}
\toprule

Case
& $d_v^c$
& $\beta_{\mathrm{pack},\mathcal{H}}$
& $\beta_{\mathrm{pack},\mathcal{G}}$
& $\gamma_{\mathrm{pack},\mathcal{H}}$
& $\gamma_{\mathrm{pack},\mathcal{G}}$
& $A_{\mathcal{H},\mathrm{sat}}$
& $A_{\mathcal{G},\mathrm{sat}}$
& $\Delta A_{\mathrm{sat}}$
\\

\midrule

$\sigma''(\bar x)=0$
& 1.11
& 0.0
& 0.0
& 2.29e-4
& 0.0
& 6.82e-4
& 6.82e-4
& 7.68e-9
\\

$\sigma''(\bar x)\neq 0$
& 1.11
& 7.86e-4
& 0.0
& 1.10e-3
& 0.0
& 1.80e-3
& 1.80e-3
& 6.41e-5
\\
\bottomrule
\end{tabular}
\end{table}

\subsection{Reduced Dynamics versus the Full System}

As an additional validation of the weakly nonlinear reduction, we compare the full projected amplitude $A_{\mathrm{num}}(t)$ with the solution of the reduced amplitude equation
\begin{equation*}
    \dot A= \mu A+\beta A^2+\gamma A^3,
\end{equation*}

A representative comparison is shown in  Figure~\ref{fig:appendix_reduction_validation}. The close agreement near onset confirms that the reduced model accurately captures both the transient evolution and the nonlinear saturation of the critical mode in the weakly nonlinear regime.

\section{Discussion and outlook}\label{sec:section7}

The main contribution of this work is a theory of structural visibility for dynamical systems on  directed hypergraphs through the lens of pattern formation. Although hypergraphs encode complex higher-order interactions, our results show that the dynamics does not access all structural information equally. Instead, different asymptotic orders reveal different  features of  hypergraph organization.

At linear level, the dynamics depends only on \emph{exposure} information. The exposure reduction theorem shows that the linearized operator probes  only exposure statistics, implying that higher-order co-occurrence information is invisible to instability onset. Consequently, hypergraphs with substantially different higher-order structure may nevertheless exhibit identical linear behavior.
Beyond onset, nonlinear saturation depends on higher-order co-occurrence structure encoded by packing tensors. 
The reduced dynamics does not depend directly on these raw tensors themselves, but on \emph{packing effects}  obtained through contraction with the critical left and right eigenvectors. The nonlinear structural decomposition theorem splits the reduced amplitude coefficients into exposure and packing effects. The visibility hierarchy then identifies which tail moments enter each asymptotic order. This leads to the notion of nonlinear distinguishability, characterizing when  linearly indistinguishable higher-order systems may exhibit different saturated amplitudes, branch selection and pattern morphology. Thus,  nonlinear saturation acts as a structural probe capable of detecting information invisible to linearization. This yields a hierarchy of structural visibility in which \emph{exposure} information determines onset, and \emph{packing effects} determine saturation and pattern selection.

A central implication is that higher-order interactions need not be dynamically visible. 
Although a hypergraph may contain substantial higher-order structure, the reduced dynamics depend only on  packing effects. Consequently, higher-order observability is projection-dependent rather than purely structural. 

The mechanism also reveals two distinct sources of dynamical invisibility. The first is \emph{structural invisibility}, which arises when the relevant packing effects vanish because of the organization of the hypergraph itself. The second is \emph{functional invisibility}, which arises when the local curvature of the activation interaction  nonlinearity suppresses higher-order effects independently of topology. In particular, if
$ D^2\sigma(\bar{x})=0, $
then the quadratic packing effect vanishes even in the presence of genuine higher-order interactions.

These results therefore clarify the precise sense in which clique expansion succeeds or fails for higher-order dynamical systems. Clique expansion is exact at linear order  because linearization probes only first tail-moment statistics, namely exposure  information.  This leads naturally to the notion of \emph{dynamical graph surrogacy} which formalizes the fact that graph reductions succeed near onset precisely when the relevant packing effects vanish.

The directed nature of the hypergraphs also plays an essential role. The linearized dynamics is  governed by generally non-symmetric exposure operators, requiring a  decomposition into left and right eigenvectors. In the weakly nonlinear regime, the adjoint eigenvector determines how nonlinear forcing projects onto the unstable mode, so directedness enters directly into the saturation mechanism through contractions with the packing tensors.

Another notable feature of the theory is the emergence of two distinct normal-form regimes. In the generic resonant case $ \beta\neq0, $
saturation occurs already at quadratic order, whereas the nonresonant case $ \beta=0 $
recovers the classical cubic Stuart--Landau equation. Thus, hypergraph-driven amplitude dynamics generically produces quadratic mode interactions unless suppressed by symmetry or projection.

Several future directions remain open. The present analysis focuses on codimension-one steady bifurcations and scalar amplitude equations. Extending the visibility framework to Hopf bifurcations, multimode interactions, and strongly nonlinear regimes may reveal additional structural observables and richer notions of distinguishability. Another promising direction is the inverse problem of structural inference. More specifically,  given nonlinear observations of a dynamical process, what aspects of the underlying higher-order structure can be reconstructed? Since  exposure-equivalent hypergraphs
are indistinguishable at linear onset, recovering higher-order structure
may require nonlinear transient or saturation data. Lastly,  control of higher-order dynamics. Since nonlinear saturation
depends on packing effects, these quantities may provide
natural targets for influencing post-onset behavior.

\vspace{2em}

\textbf{Acknowledgements:} This work was partially funded by NSF grant  DMS-2103026, and AFOSR grants FA
9550-22-1-0215 and FA 9550-23-1-0400


\appendix

\renewcommand{\thesection}{\Alph{section}}
\renewcommand{\theequation}{\thesection.\arabic{equation}}
\numberwithin{equation}{section}
\renewcommand{\thefigure}{\Alph{section}.\arabic{figure}}
\numberwithin{figure}{section}
\numberwithin{table}{section}

\section{Aggregator-Mixer Coupling Examples}
\label{app:aggr-mixer}

\subsection{Multilinear Coupling}

A particularly simple and analytically tractable subclass of the aggregator--mixer framework is obtained by choosing linear aggregation functions. In this setting, the aggregators $\psi_k$ are taken to be the arithmetic mean of the tail states:
\begin{equation}
\psi_k(x_1,\dots,x_k)
=
\frac{1}{k}\sum_{\ell=1}^{k}x_\ell
\end{equation}

Choosing the mixer to be the identity map, $\Theta_k(x)=x$,  the resulting coupling function becomes
\begin{equation}\label{linear-coupling}
\phi_k(x_{j_1},\dots,x_{j_k},x_i) =
M_k
\left(
\frac{\beta_k}{k}
\sum_{\ell=1}^{k}x_{j_\ell}
+
\alpha_k x_i
\right)
\end{equation}

This coupling structure generalizes classical linear diffusive interactions on graphs to directed hypergraph settings. In particular, the interaction term depends linearly on the average state of the tail nodes together with a linear contribution from the head node. Such couplings provide a natural higher-order extension of consensus, diffusion, and linear network interaction models.

\subsection{Polynomial Coupling}

An important subclass of the aggregator--mixer framework is obtained by choosing polynomial aggregation functions. In this setting, the $i$-th canonical projection of the aggregator $\psi_k$ is taken to be a monomial of degree at least $k$. More precisely, letting $\{\mathbf e_i\}_{1\le i\le d}$ denote the canonical basis of $\mathbb R^d$, we assume
\begin{equation}
\mathbf e_i^{T}\psi_k(x_1,\dots,x_k)
=
\prod_{\ell=1}^{k}
x_{\ell,i}^{\,m_\ell},
\qquad
\sum_{\ell=1}^{k}m_\ell \ge k,
\qquad
m_\ell \ge 1,
\end{equation}
where $x_{\ell,i}$ denotes the $i$-th component of $x_\ell$

Equivalently, the aggregation function may be expressed using the Hadamard product:
\begin{equation}
\psi_k(x_1,\dots,x_k)
=
\odot_{\ell=1}^{k}
\odot_{\nu=1}^{m_\ell}
x_\ell 
\end{equation}

Choosing the mixer function to be the identity,
$
\Theta_k=\mathrm{Id}
$
the coupling function reduces to the polynomial interaction form
\begin{equation}\label{polynomial-coupling}
\phi_k(x_{j_1},\dots,x_{j_k},x_i)
=
M_k
\left(
\beta_k
\odot_{\ell=1}^{k}
\odot_{\nu=1}^{m_\ell}
x_{j_\ell}
+
\alpha_k
\odot_{\ell=1}^{k}
\odot_{\nu=1}^{m_\ell}
x_i
\right)
\end{equation}

\subsection{Homogeneous Polynomial Coupling}

In  the \emph{Polynomial Coupling} above, the aggregators $\psi_k$, and consequently the coupling functions $\phi_k$ need not have the same degree for all $k\in \{1,\cdots,\rho-1\}$. This leads in general  to \emph{non-homogeneous} polynomial coupling term in \eqref{eq:hypergraph-dynamics}. However, if  instead one seeks \emph{homogeneous polynomial} coupling structure, then additional constraints must be imposed on the aggregators $\psi_k$. More specifically,  all aggregators, and therefore all the coupling functions $\phi_k$ must be chosen to have \emph{the same degree} $d$. Note that the  degree must be \emph{at least} $\rho-1$, namely $d \ge \rho-1$, where $\rho = r(\mathcal{H})$ is the \emph{rank} of the hypergraph. Table \ref{homogeneous-polynomial-coupling-table} shows an example of the \emph{multiindice} exponents $(m_1,\cdots,m_k)$ to realize a homogeneous polynomial coupling in the dynamics \eqref{eq:hypergraph-dynamics}.
This class of coupling includes multilinear tensor interactions in \cite{chen2021controllability} as a particular case, and it provides a natural higher-order generalization of homogeneous polynomial coupling mechanisms commonly used in tensor-based dynamical systems.

\begin{table}[H]
\centering

\caption{
Example of multi-index exponents
$(m_1,\dots,m_k)\in\mathbb{N}^k$,
where $\mathbb{N}=\{1,2,\dots\}$ and
$1\le k\le \rho-1$.
These choices yield homogeneous polynomial coupling terms in the dynamics
\eqref{eq:hypergraph-dynamics}.
The minimum degree of the aggregators $\psi_k$ is $\rho-1$, where $\rho$ denotes the rank of the hypergraph.
Multiplication is interpreted componentwise via the Hadamard product since the variables $x_i$ may be high-dimensional.
}

\label{homogeneous-polynomial-coupling-table}
\renewcommand{\arraystretch}{1.2}
\setlength{\tabcolsep}{8pt}
\begin{tabular}{l c c c}
\toprule

 $\psi_k$
& $\deg(\psi_k)=\rho-1$
& $\deg(\psi_k)=\rho$
& {$\dots$}

\\

\midrule

$\psi_1(x_1)$
& {$x_1^{\rho-1}$}
& {$x_1^{\rho}$}
& {$\dots$}

\\

$\psi_2(x_1,x_2)$
& {$x_1^{\rho-2}x_2$}
& {$x_1^{\rho-1}x_2$}
& {$\dots$}

\\

$\vdots$
& {$\vdots$}
& {$\vdots$}
& {$\vdots$}

\\

$\psi_{\rho-1}(x_1,\dots,x_{\rho-1})$
& {$x_1x_2\cdots x_{\rho-1}$}
& {$x_1^2x_2\cdots x_{\rho-1}$}
& {$\dots$}
\\
\bottomrule
\end{tabular}
\end{table}

\subsection{Trigonometric Coupling}\label{app:trigonometric-coupling}

Another important subclass of the aggregator-mixer framework is obtained by combining multilinear aggregation with trigonometric mixing functions. In this setting, the aggregation function $\psi_k$ is chosen to be multilinear, while the canonical projections of the mixer $\Theta_k$ are taken to be trigonometric functions such as \emph{sine} or \emph{cosine}.
A particularly relevant example arises when choosing the aggregation function to be the summation
\begin{equation}
\psi_k(x_1,\dots,x_k)
=
\sum_{\ell=1}^{k}x_\ell
\end{equation}
and selecting the mixer so that its canonical projections are given by the sine function,
\begin{equation}
\mathbf e_i^{T}\Theta_k(x)
=
\sin(x_i),
\qquad
1\le i\le d.
\end{equation}
This yields a higher-order trigonometric coupling structure generalizing classical Kuramoto-type dynamics.
To illustrate this connection, suppose for simplicity that $\mathcal H$ is a $3$-uniform directed B-hypergraph and that the node dynamics are scalar-valued. The corresponding order-$2$ coupling function  on the B-arc $e= \left( \left\{ j_1,j_2 \right\} \to i \right)$ takes the form
\begin{equation}
\phi_2(x_{j_1},x_{j_2},x_i)
=
\alpha_2
\sin\!\big(
x_{j_1}+x_{j_2}-2x_i
\big)
\end{equation}

This expression is a direct higher-order extension of the classical Kuramoto interaction term, with pairwise phase differences replaced by collective phase interactions involving multiple tail nodes. The normalization factor is absorbed into the adjacency tensor
$
\mathsf A^{(2)}_{i j_1 j_2}.
$
 Consequently, the framework naturally incorporates higher-order Kuramoto-type models on directed hypergraphs within the general aggregator-mixer architecture.

\section{Additional  Derivations}
\label{app:add-derivations}

\subsection{Proof of Lemma~\ref{lem:symmetry-and-contraction-invariance}}\label{app:der-lem-sym-and-contraction-invariance}

\begin{proof}
Fix
$i\in\{1,\dots,n_0\}$ and  $j\in\{1,\dots,n\}$.
For each $\ell\in\{1,\dots,k\} $, we define the following  index set
$\mathcal I_{\ell}(j):=\left\{(j_1,\dots,j_k)\in\{1,\dots,n\}^k
\;\middle|\;
j_\ell=j
\right\}.
$
Then by definition of the contraction we have, $
\mathsf{T}^{\hat{\ell}}_{ij} = \sum_{(j_1,\dots,j_k)\in\mathcal I_{\ell}(j)} \mathsf{T}_{i j_1\dots j_k}$.
Now let $\ell,\ell'\in\{1,\dots,k\}$ , and consider the permutation $ \pi_{\ell\leftrightarrow\ell'}\in S_k $
that swaps the positions $\ell$ and $\ell'$. This permutation induces a bijection between the sets  $\mathcal I_{\ell}(j)$ and $\mathcal I_{\ell'}(j)$ since fixing the value $j$ in the $\ell$-th position is equivalent after permutation to fixing $j$ in the $\ell'$-th position.
By symmetry of the tensor $\mathsf{T}$ in its last $k$ indices, we have
$ \mathsf{T}_{i j_1\dots j_k} = \mathsf{T}_{i j_{\pi(1)}\dots j_{\pi(k)}} $
for every permutation $\tau\in S_k$. In particular, for the permutation $\pi$ above, corresponding terms under the above bijection have identical tensorial value, namely,  $ \mathsf{T}^{\hat{\ell}}_{ij}= \mathsf{T}^{\hat{\ell}'}_{ij}$.
Since the equality holds for all $i\in\{1,\dots,n_0\}$ and  $j\in\{1,\dots,n\}$,  the result follows.
\end{proof}

\subsection{Proof of Proposition~\ref{prop:pth-tail-moment-as-tensor-contraction}}\label{app:der-prop-p-tail-moments-as-tensor-contraction}

\begin{proof}
    Let $j_1,\cdots, j_p$ be fixed tail indices, and consider a hyperedge $e = (T\to i) \in E_k$ such that $j_1,\cdots, j_p \in T$. Using the definition of the adjacency tensor \eqref{eq:adjacency-tensor-def}, the contribution of $e$ to the contraction $\left\langle
\mathsf{A}^{(k)},
\mathsf{1}_{k-p}
\right\rangle_{\{p+2,\dots,k+1\}}$ is $ (k-p)!\,\frac{w(e)}{k!} = \frac{w(e)}{(k)_p}$. Thus multiplying by $(k)_p/k^p$ and summing over all hyperedges containing $j_1,\cdots,j_p$ yields the result. Note that the restriction to distinct tail indices isolates the genuinely $p$-way interactions of $\mathsf T^{(k,p)}$, since repeated-index entries reduce to lower-order moments.
\end{proof}

\subsection{Proof of Proposition~\ref{prop:center-manifold-reduction-n-amplitude-eqn}}\label{app:der-prop-center-manifold-reduction}

\begin{proof}
Let $y=\delta \mathbf{x}$ and consider the following cubic order Taylor expansion of the vector field near the bifurcation
\begin{align}\label{app:vector-field-expansion-center-manifold}
\dot y
=
J_0y+\epsilon J_1y+B(y,y)+C(y,y,y)
+
O(\|y\|^4+|\epsilon|\|y\|^2+\epsilon^2\|y\|),
\end{align}
where $B$ and $C$ are symmetric bilinear and trilinear maps. Since
$J_0$ has a simple zero eigenvalue with right eigenvector $v$ and left
eigenvector $w$ satisfying $w^\top v=1$, with all remaining
eigenvalues have negative real part, the center manifold theorem, guarantees the existence of a one-dimensional locally
invariant center manifold locally tangent to $\operatorname{span}\{v\}$ at the
origin. Let's  parametrize it as follows
\begin{align}\label{app:center-manifold-eq}
y=Av+h(A,\epsilon),
\end{align}
where $A\in\mathbb R$ and $h(A,\epsilon)$ takes values in the stable
complement of $\operatorname{span}\{v\}$. Let $Q=I-vw^\top $
denote the projection onto this stable complement. We have the following expansion of the correction term owing to the tangency of the center manifold to  $\operatorname{span}\{v\}$

\begin{align}\label{app:correction-term-center-manifold-eq}
h(A,\epsilon)=A^2u+O(|A|^3+|\epsilon||A|^2)
\end{align}
for some $u\in\operatorname{Range}(Q)$. The amplitude is defined using the adjoint vector $w$ such that $A = w^\top y$. Since the correction term is transverse to the critical subspace, namely $w^\top h(A,\epsilon)$, differentiating the projection of \eqref{app:center-manifold-eq} yields $
\dot A=w^\top \dot y$.
Substituting the expansion of \eqref{app:center-manifold-eq} into  \eqref{app:vector-field-expansion-center-manifold} and keeping only terms
up to cubic order, we obtain
\begin{align}\label{app:expansion-A-dot-center-manifold}
\dot A = w^\top\Big[
J_0(Av+A^2u)
+\epsilon J_1(Av)
+B(Av+A^2u,Av+A^2u)
+C(Av,Av,Av)
\Big]
+\text{h.o.t.}
\end{align}
Since  $J_0v=0$ and $w^\top J_0=0$, we obtain
\begin{align}
\dot A = \epsilon\, w^\top J_1v\, A
+
w^\top B(v,v)\,A^2
+
w^\top\Big(2B(v,u)+C(v,v,v)\Big)\,A^3
+\text{h.o.t.}
\end{align}
Hence
\begin{align}
\alpha=w^\top J_1v,\qquad
\beta=w^\top B(v,v),\qquad
\gamma=w^\top\Big(2B(v,u)+C(v,v,v)\Big)
\end{align}

It now remains to determine  $u$. In the one hand, substituting the expansion of the ansatz \eqref{app:center-manifold-eq} into \eqref{app:vector-field-expansion-center-manifold} then applying the stable projection $Q$ gives $J_0u+QB(v,v)$ at order $A^2$. In the other hand, differentiating the expansion of the ansatz \eqref{app:center-manifold-eq}, then applying the stable projection Q yields $Q\dot{y}= 2A\dot{A}u + O(A^2\dot{A}) = O(A^3)$ since $\dot{A} = O(A^2)$ at the bifurcation point. Thus the center manifold invariance
condition requires the quadratic component $J_0u+QB(v,v)$ to vanish.
Since $u\in\operatorname{Range}(Q)$, the center-manifold invariance
Since $u\in\operatorname{Range}(Q)$, the center-manifold invariance
condition requires this component  to vanish. Since the zero eigenvalue of $J_0$ is simple, its restriction to the stable subspace $\operatorname{Range}(Q)$ is invertible. Thus the equation has a unique solution.
\end{proof}

\section{Numerical Validation}
\label{app:num-val}

In this appendix section, we present the benchmark hypergraphs used throughout the numerical experiments in Chapter~\ref{sec:section6}, and we compare the dynamics of the full hypergraph system with that predicted by the reduced amplitude equation. These findings validate the accuracy of the weakly nonlinear reduction near the instability onset.

\subsection{Hypergraph Benchmark Sets}\label{app:num-val:benchmark-pairs}

The numerical experiments are performed on a set of small directed weighted hypergraphs selected to demonstrate the main theoretical phenomena. Hypergraphs $\mathcal{H}_A$ and $\mathcal{H}_B$ form an exposure-equivalent pair, and they are used to illustrate nonlinear distinguishability despite identical linear onset behavior. The hypergraph $\mathcal{H}$ serves as a representative benchmark for validating dynamical graph surrogacy and assessing  the accuracy of the reduced dynamics against direct numerical simulations of the full system.

\begin{table}[H]
\centering
\begin{minipage}{0.47\textwidth}
\centering
\[
V=\{1,2,3,4,5,6\}
\]
\begin{tabular}{c c c}
\hline
Hyperedge & Directed relation & Weight \\
\hline
$e_1$ & $\{3,4\}\to1$ & $2.00$ \\
$e_2$ & $\{5\}\to1$ & $0.60$ \\
$e_3$ & $\{5\}\to2$ & $0.80$ \\
$e_4$ & $\{6\}\to2$ & $0.80$ \\
$e_5$ & $\{3\}\to2$ & $0.40$ \\
$e_6$ & $\{1\}\to3$ & $0.50$ \\
$e_7$ & $\{2\}\to4$ & $0.50$ \\
\hline
\end{tabular}
\caption*{(a) Hypergraph $\mathcal{H}_A$}
\end{minipage}
\hfill
\begin{minipage}{0.47\textwidth}
\centering
\[
V=\{1,2,3,4,5,6\}
\]
\begin{tabular}{c c c}
\hline
Hyperedge & Directed relation & Weight \\
\hline
$e_1$ & $\{3\}\to1$ & $1.00$ \\
$e_2$ & $\{4\}\to1$ & $1.00$ \\
$e_3$ & $\{5\}\to1$ & $0.60$ \\
$e_4$ & $\{5,6\}\to2$ & $1.60$ \\
$e_5$ & $\{3\}\to2$ & $0.40$ \\
$e_6$ & $\{1\}\to3$ & $0.50$ \\
$e_7$ & $\{2\}\to4$ & $0.50$ \\
\hline
\end{tabular}
\caption*{(b) Hypergraph $\mathcal{H}_B$}
\end{minipage}
\caption{Pair  of directed weighted hypergraphs $\mathcal{H}_A$ and $\mathcal{H}_B$ used in Figure~\ref{fig:pairI_distinguishability}}
\label{tab:HA_HB}
\end{table}

\begin{table}[H]
\centering
\[
V=\{1,2,3,4,5,6\}
\]
\begin{tabular}{c c c}
\hline
Hyperedge & Directed relation & Weight \\
\hline
$e_1$ & $\{3,4\}\to1$ & $1.01$ \\
$e_2$ & $\{5,6\}\to1$ & $1.00$ \\
$e_3$ & $\{3,6\}\to2$ & $0.80$ \\
$e_4$ & $\{3\}\to5$ & $1.50$ \\
$e_5$ & $\{5\}\to1$ & $0.60$ \\
$e_6$ & $\{1\}\to3$ & $0.50$ \\
$e_7$ & $\{5\}\to2$ & $0.75$ \\
$e_8$ & $\{1,2,5\}\to4$ & $0.50$ \\
$e_9$ & $\{4,5,6\}\to1$ & $0.50$ \\
\hline
\end{tabular}
\caption{Directed weighted hypergraph $\mathcal{H}$ used in Figure~\ref{fig:graph_surrogacy} and Figure~\ref{fig:appendix_reduction_validation}.}
\label{tab:H3}
\end{table}

\subsection{Reduced Dynamics Versus  Full System}\label{app:num-val_red-dyn-vs-full-sys}

To assess the accuracy of the weakly nonlinear reduction, we compare the amplitude predicted by the reduced dynamics \eqref{eq:amplitude-equation} with the projection of the full system (\eqref{eq:hypergraph-dynamics} + \eqref{eq:deepSet-coupling-function}). Simulations are performed for parameter values close to the instability threshold, where the assumptions underlying the reduction are expected to hold. The comparison focuses on both the transient growth phase and the eventual saturation amplitude.

\begin{figure}[H]
\centering
\includegraphics[width=.95\textwidth]{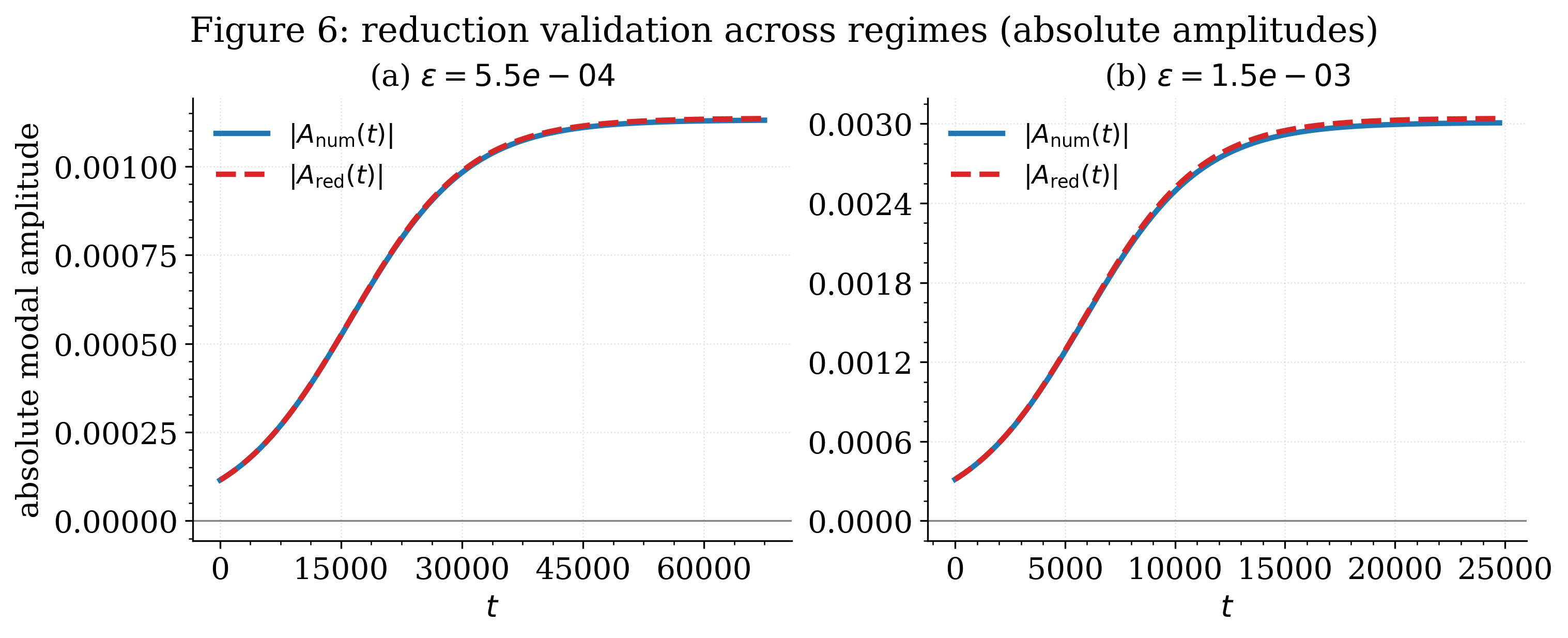}
\caption[Reduction validation across regimes]{
Reduction validation across regimes.
The projected amplitude
$A_{\mathrm{num}}(t)$
from the full hypergraph system is compared with the solution
$A_{\mathrm{red}}(t)$
of the reduced normal form, initialized with
$A_{\mathrm{red}}(0)=A_{\mathrm{num}}(0)$.
\textbf{Left:}
Near-onset regime with
$\varepsilon=5.5\times 10^{-4}$.
The reduced equation accurately reproduces both the transient growth and the saturation amplitude over long timescales, consistent with the validity of the weakly nonlinear reduction close to bifurcation.
\textbf{Right:}
Moderately larger parameter offset with
$\varepsilon=1.5\times 10^{-3}$.
Although the dynamics move farther from the asymptotic weakly nonlinear regime, the reduced model still captures the qualitative evolution and provides a quantitatively accurate approximation of the saturation dynamics.
In both cases, the close agreement confirms that the reduced amplitude equation captures the dominant critical-mode dynamics governing pattern formation near onset.
}
\label{fig:appendix_reduction_validation}
\end{figure}


\bibliographystyle{unsrt}
\bibliography{references}
\end{document}

%% file: preamble.tex
\usepackage{amsmath,amssymb,amsfonts,amsthm,mathtools}
\usepackage{bbm,dsfont}

\usepackage{graphicx}
\usepackage{subcaption}
\usepackage{booktabs}
\usepackage{array}
\usepackage{float}
\usepackage[skip=3pt]{caption}

\usepackage[dvipsnames]{xcolor}

\usepackage{tikz}
\usetikzlibrary{calc,topaths}
\usetikzlibrary{positioning}
\usetikzlibrary{arrows.meta}

\usepackage{pgfplots}
\pgfplotsset{compat=1.18}

\usepackage{enumitem}

\usepackage[linesnumbered,ruled,vlined]{algorithm2e}
\SetKwInput{KwInput}{Input}
\SetKwInput{KwOutput}{Output}

\usepackage{hyperref}

\hypersetup{
    colorlinks=true,
    linkcolor=blue,
    citecolor=blue,
    urlcolor=cyan
}

\usepackage[nameinlink,noabbrev]{cleveref}

\newtheorem{theorem}{Theorem}[section]
\newtheorem{proposition}[theorem]{Proposition}
\newtheorem{lemma}[theorem]{Lemma}
\newtheorem{corollary}[theorem]{Corollary}

\theoremstyle{definition}
\newtheorem{definition}[theorem]{Definition}
\newtheorem{example}[theorem]{Example}

\theoremstyle{remark}
\newtheorem{remark}[theorem]{Remark}

\newtheoremstyle{named}
  {\topsep}
  {\topsep}
  {\normalfont}
  {}
  {\bfseries}
  {.}
  { }
  {\thmname{#1}\thmnumber{ #2}\thmnote{ (#3)}}

\theoremstyle{named}


%% file: references.bib
@article{battiston2020networks,
  title={Networks beyond pairwise interactions: Structure and dynamics},
  author={Battiston, Federico and Cencetti, Giulia and Iacopini, Iacopo and Latora, Vito and Lucas, Maxime and Patania, Alice and Young, Jean-Gabriel and Petri, Giovanni},
  journal={Physics reports},
  volume={874},
  pages={1--92},
  year={2020},
  doi = {10.1016/j.physrep.2020.05.004},
  publisher={Elsevier}
}

@article{battiston2021physics,
  title={The physics of higher-order interactions in complex systems},
  author={Battiston, Federico and Amico, Enrico and Barrat, Alain and Bianconi, Ginestra and Ferraz de Arruda, Guilherme and Franceschiello, Benedetta and Iacopini, Iacopo and K{\'e}fi, Sonia and Latora, Vito and Moreno, Yamir and others},
  journal={Nature physics},
  volume={17},
  number={10},
  pages={1093--1098},
  year={2021},
  doi={10.1038/s41567-021-01371-4},
  publisher={Nature Publishing Group UK London}
}

@article{klamt2009hypergraphs,
  title={Hypergraphs and cellular networks},
  author={Klamt, Steffen and Haus, Utz-Uwe and Theis, Fabian},
  journal={PLoS computational biology},
  volume={5},
  number={5},
  pages={e1000385},
  year={2009},
  doi={10.1371/journal.pcbi.1000385},
  publisher={Public Library of Science San Francisco, USA}
}

@article{iacopini2019simplicial,
  title={Simplicial models of social contagion},
  author={Iacopini, Iacopo and Petri, Giovanni and Barrat, Alain and Latora, Vito},
  journal={Nature communications},
  volume={10},
  number={1},
  pages={2485},
  year={2019},
  doi={10.1038/s41467-019-10431-6},
  publisher={Nature Publishing Group UK London}
}

@article{bairey2016high,
  title={High-order species interactions shape ecosystem diversity},
  author={Bairey, Eyal and Kelsic, Eric D and Kishony, Roy},
  journal={Nature communications},
  volume={7},
  number={1},
  pages={12285},
  year={2016},
  doi={10.1038/ncomms12285},
  publisher={Nature Publishing Group UK London}
}

@book{centola2018behavior,
  title={How behavior spreads: The science of complex contagions},
  author={Centola, Damon},
  volume={3},
  year={2018},
  doi={10.2307/j.ctvc7758p},
  publisher={Princeton University Press Princeton, NJ}
}

@book{bianconi2021higher,
  title={Higher-order networks},
  author={Bianconi, Ginestra},
  year={2021},
  doi={10.1017/9781108770996},
  publisher={Cambridge University Press}
}

@article{lucas2020multiorder,
  title={Multiorder Laplacian for synchronization in higher-order networks},
  author={Lucas, Maxime and Cencetti, Giulia and Battiston, Federico},
  journal={Physical Review Research},
  volume={2},
  number={3},
  pages={033410},
  year={2020},
  doi={10.1103/PhysRevResearch.2.033410},
  publisher={APS}
}

@article{schaub2020random,
  title={Random walks on simplicial complexes and the normalized Hodge 1-Laplacian},
  author={Schaub, Michael T and Benson, Austin R and Horn, Paul and Lippner, Gabor and Jadbabaie, Ali},
  journal={SIAM Review},
  volume={62},
  number={2},
  pages={353--391},
  year={2020},
  doi={10.1137/18M1201019},
  publisher={SIAM}
}

@article{millan2020explosive,
  title={Explosive higher-order Kuramoto dynamics on simplicial complexes},
  author={Mill{\'a}n, Ana P and Torres, Joaqu{\'\i}n J and Bianconi, Ginestra},
  journal={Physical Review Letters},
  volume={124},
  number={21},
  pages={218301},
  year={2020},
  doi={10.1103/PhysRevLett.124.218301},
  publisher={APS}
}

@inproceedings{turing1952chemical,
  title={The chemical basis of morphogenesis: Philosophical transactions of the roy al society of london. ser. b},
  author={Turing, AM},
  booktitle={Biol. Sci},
  volume={237},
  doi={10.1098/rstb.1952.0012},
  year={1952}
}

@article{nakao2010turing,
  title={Turing patterns in network-organized activator--inhibitor systems},
  author={Nakao, Hiroya and Mikhailov, Alexander S},
  journal={Nature Physics},
  volume={6},
  number={7},
  pages={544--550},
  year={2010},
  doi={10.1038/nphys1651},
  publisher={Nature Publishing Group UK London}
}

@article{carletti2020dynamical,
  title={Dynamical systems on hypergraphs},
  author={Carletti, Timoteo and Fanelli, Duccio and Nicoletti, Sara},
  journal={Journal of Physics: Complexity},
  volume={1},
  number={3},
  pages={035006},
  year={2020},
  doi={10.1088/2632-072X/aba8e1},
  publisher={IOP Publishing}
}

@article{carletti2020random,
  title={Random walks on hypergraphs},
  author={Carletti, Timoteo and Battiston, Federico and Cencetti, Giulia and Fanelli, Duccio},
  journal={Physical review E},
  volume={101},
  number={2},
  pages={022308},
  year={2020},
  doi={10.1103/PhysRevE.101.022308},
  publisher={APS}
}

@book{murray2003mathematical,
  title={Mathematical biology: II: spatial models and biomedical applications},
  author={Murray, James Dickson and Murray, James Dickson},
  volume={18},
  year={2003},
  doi={10.1007/b98869},
  publisher={Springer}
}

@article{cross1993pattern,
  title={Pattern formation outside of equilibrium},
  author={Cross, Mark C and Hohenberg, Pierre C},
  journal={Reviews of modern physics},
  volume={65},
  number={3},
  pages={851},
  year={1993},
  doi={10.1103/RevModPhys.65.851},
  publisher={APS}
}

@book{hoyle2006pattern,
  title={Pattern formation: an introduction to methods},
  author={Hoyle, Rebecca B},
  year={2006},
  doi={10.1017/CBO9780511616051},
  publisher={Cambridge University Press}
}

@article{maini2012turing,
  title={Turing's model for biological pattern formation and the robustness problem},
  author={Maini, Philip K and Woolley, Thomas E and Baker, Ruth E and Gaffney, Eamonn A and Lee, S Seirin},
  journal={Interface focus},
  volume={2},
  number={4},
  pages={487},
  doi={10.1098/rsfs.2011.0113},
  year={2012}
}

@article{zaheer2017deep,
  title={Deep sets},
  author={Zaheer, Manzil and Kottur, Satwik and Ravanbakhsh, Siamak and Poczos, Barnabas and Salakhutdinov, Russ R and Smola, Alexander J},
  journal={Advances in neural information processing systems},
  volume={30},
  year={2017}
}

@inproceedings{lee2019set,
  title={Set transformer: A framework for attention-based permutation-invariant neural networks},
  author={Lee, Juho and Lee, Yoonho and Kim, Jungtaek and Kosiorek, Adam and Choi, Seungjin and Teh, Yee Whye},
  booktitle={International conference on machine learning},
  pages={3744--3753},
  year={2019},
  organization={PMLR}
}

@article{gallo1993directed,
  title={Directed hypergraphs and applications},
  author={Gallo, Giorgio and Longo, Giustino and Pallottino, Stefano and Nguyen, Sang},
  journal={Discrete applied mathematics},
  volume={42},
  number={2-3},
  pages={177--201},
  year={1993},
  doi={10.1016/0166-218X(93)90045-P},
  publisher={Elsevier}
}

@inproceedings{mouyebe2025coupling,
  title={Coupling Induced Stabilization of Network Dynamical Systems and Switching},
  author={Mouyebe, Moise R and Bloch, Anthony M},
  booktitle={2025 IEEE 64th Conference on Decision and Control (CDC)},
  pages={8091--8097},
  year={2025},
  doi={10.1109/CDC57313.2025.11312567},
  organization={IEEE}
}

@article{olfati2004consensus,
  title={Consensus problems in networks of agents with switching topology and time-delays},
  author={Olfati-Saber, Reza and Murray, Richard M},
  journal={IEEE Transactions on automatic control},
  volume={49},
  number={9},
  pages={1520--1533},
  year={2004},
  doi={10.1109/TAC.2004.834113},
  publisher={IEEE}
}

@article{ren2005consensus,
  title={Consensus seeking in multiagent systems under dynamically changing interaction topologies},
  author={Ren, Wei and Beard, Randal W},
  journal={IEEE Transactions on automatic control},
  volume={50},
  number={5},
  pages={655--661},
  year={2005},
  doi={10.1109/TAC.2005.846556},
  publisher={IEEE}
}

@article{arenas2008synchronization,
  title={Synchronization in complex networks},
  author={Arenas, Alex and D{\'\i}az-Guilera, Albert and Kurths, Jurgen and Moreno, Yamir and Zhou, Changsong},
  journal={Physics reports},
  volume={469},
  number={3},
  pages={93--153},
  year={2008},
  doi={10.1016/j.physrep.2008.09.002},
  publisher={Elsevier}
}

@inproceedings{kuramoto2005self,
  title={Self-entrainment of a population of coupled non-linear oscillators},
  author={Kuramoto, Yoshiki},
  booktitle={International symposium on mathematical problems in theoretical physics: January 23--29, 1975, kyoto university, kyoto/Japan},
  pages={420--422},
  year={2005},
  doi={10.1007/bfb0013365},
  organization={Springer}
}

@article{acebron2005kuramoto,
  title={The Kuramoto model: A simple paradigm for synchronization phenomena},
  author={Acebr{\'o}n, Juan A and Bonilla, Luis L and P{\'e}rez Vicente, Conrad J and Ritort, F{\'e}lix and Spigler, Renato},
  journal={Reviews of modern physics},
  volume={77},
  number={1},
  pages={137--185},
  year={2005},
  doi={10.1103/RevModPhys.77.137},
  publisher={APS}
}

@article{strogatz2000kuramoto,
  title={From Kuramoto to Crawford: exploring the onset of synchronization in populations of coupled oscillators},
  author={Strogatz, Steven H},
  journal={Physica D: Nonlinear Phenomena},
  volume={143},
  number={1-4},
  pages={1--20},
  year={2000},
  doi={10.1016/S0167-2789(00)00094-4},
  publisher={Elsevier}
}

@article{chen2021controllability,
  title={Controllability of hypergraphs},
  author={Chen, Can and Surana, Amit and Bloch, Anthony M and Rajapakse, Indika},
  journal={IEEE Transactions on Network Science and Engineering},
  volume={8},
  number={2},
  pages={1646--1657},
  year={2021},
  doi={10.1109/TNSE.2021.3068203},
  publisher={IEEE}
}

@article{muolo2023turing,
  title={Turing patterns in systems with high-order interactions},
  author={Muolo, Riccardo and Gallo, Luca and Latora, Vito and Frasca, Mattia and Carletti, Timoteo},
  journal={Chaos, Solitons \& Fractals},
  volume={166},
  pages={112912},
  year={2023},
  doi={10.1016/j.chaos.2022.112912},
  publisher={Elsevier}
}

@article{kolda2009tensor,
  title={Tensor decompositions and applications},
  author={Kolda, Tamara G and Bader, Brett W},
  journal={SIAM review},
  volume={51},
  number={3},
  pages={455--500},
  year={2009},
  doi={10.1137/07070111X},
  publisher={SIAM}
}

@article{horn1945matrix,
  title={Matrix Analysis},
  author={Horn, Roger A and Johnson, Charles R},
  journal={American history},
  volume={1861},
  number={1900},
  doi={10.1017/CBO9781139020411},
  year={1945}
}

@book{golub2013matrix,
  title={Matrix computations},
  author={Golub, Gene H and Van Loan, Charles F},
  year={2013},
  doi={10.56021/9781421407944},
  publisher={JHU press}
}

@book{trefethen1997numerical,
  title={Numerical Linear Algebra},
  author={Trefethen, Lloyd N and Bau, III, David},
  year={1997},
  doi={10.1137/1.9780898719574},
  publisher={SIAM}
}

@book{carr2012applications,
  title={Applications of centre manifold theory},
  author={Carr, Jack},
  year={2012},
  doi={10.1007/978-1-4612-5929-9},
  publisher={Springer Science \& Business Media}
}

@book{kuznetsov1998elements,
  title={Elements of applied bifurcation theory},
  author={Kuznetsov, Yuri A},
  year={1998},
  doi={10.1007/978-3-031-22007-4},
  publisher={Springer}
}

@book{guckenheimer2013nonlinear,
  title={Nonlinear oscillations, dynamical systems, and bifurcations of vector fields},
  author={Guckenheimer, John and Holmes, Philip},
  year={2013},
  doi={10.1007/978-1-4612-1140-2},
  publisher={Springer Science \& Business Media}
}

@book{wiggins2003introduction,
  title={Introduction to applied nonlinear dynamical systems and chaos},
  author={Wiggins, Stephen},
  year={2003},
  doi={10.1007/b97481},
  publisher={Springer}
}

@article{newell1969finite,
  title={Finite bandwidth, finite amplitude convection},
  author={Newell, Alan C and Whitehead, John A},
  journal={Journal of Fluid Mechanics},
  volume={38},
  number={2},
  pages={279--303},
  year={1969},
  publisher={Cambridge University Press}
}

@book{cross2009pattern,
  title={Pattern formation and dynamics in nonequilibrium systems},
  author={Cross, Michael and Greenside, Henry},
  year={2009},
  doi={10.1017/CBO9780511627200},
  publisher={Cambridge University Press}
}

@article{schnakenberg1979simple,
  title={Simple chemical reaction systems with limit cycle behaviour},
  author={Schnakenberg, J{\"u}rgen},
  journal={Journal of theoretical biology},
  volume={81},
  number={3},
  pages={389--400},
  year={1979},
  doi={10.1016/0022-5193(79)90042-0},
  publisher={Elsevier}
}
